\documentclass[a4paper, 11pt]{scrartcl}
    \usepackage[a4paper,left=2cm,right=2cm,top=3cm,bottom=4cm]{geometry}
	\usepackage{url}	
	\usepackage{caption}
	\usepackage{enumitem} 
	\usepackage[utf8]{inputenc} 
	\usepackage{amsmath}   
	\usepackage{amsthm}    
	\usepackage{amsfonts}  
	\usepackage{graphicx}  
	\usepackage{float}	 
	\usepackage{mathrsfs} 
	\usepackage{amssymb} 
	\usepackage{mathtools} 
	\usepackage{microtype} 
	\usepackage[square, numbers]{natbib}
		\usepackage[ruled]{algorithm2e}
	\usepackage{algpseudocode} 
	\usepackage{blindtext}
	\usepackage[colorinlistoftodos]{todonotes}  
	\usepackage[hidelinks]{hyperref}
	\usepackage{cleveref}
	\usepackage{bbm} 
	
	\usepackage[utf8]{inputenc}
	\usepackage[acronym,numberedsection,section=subsection]{glossaries}
	
 \newcommand{\N}{\ensuremath{\mathbb{N}}}   
  \newcommand{\R}{\ensuremath{\mathbb{R}}}   
  \renewcommand{\epsilon}{\varepsilon}       
  \newcommand{\eps}{\varepsilon}       

  \newtheorem{satz}{Satz}[section]
  \newtheorem{lemma}[satz]{Lemma}
  
  \newtheorem{proposition}[satz]{Proposition}
  \newtheorem{theorem}[satz]{Theorem}
  \newtheorem{remark}[satz]{Remark}
  \theoremstyle{definition}
  \newtheorem{definition}[satz]{Definition}
  \newtheorem{example}[satz]{Example}
  \theoremstyle{remark}
  
\title{Dynamic optimal transport on networks}
\author{Martin Burger\thanks{Friedrich-Alexander-Universität Erlangen-Nürnberg, Department Mathematik, Germany (martin.burger@fau.de).} \ \
Ina Humpert\thanks{Westf\"alische Wilhelms-Universit\"at (WWU) M\"unster, Institute for Analysis and Computational Mathematics, Germany (ina.humpert@uni-muenster.de).} \ \
Jan-Frederik Pietschmann\thanks{Technische Universität Chemnitz, Fakult\"at f\"ur Mathematik, Germany (jfpietschmann@math.tu-chemnitz.de).} \ \
}

\usepackage{natbib}
\usepackage{graphicx}

\begin{document}

\maketitle

\begin{abstract}
	\begin{center}
		\textbf{Abstract}
	\end{center}
\vspace{1ex}
We study a dynamic optimal transport problem on a network. Despite the cost for transport along the edges, an additional cost, scaled with a parameter $\kappa$, has to be paid for interchanging mass between edges and vertices. We show existence of minimisers using duality and discuss the relationship of the model to other metrics such as Fisher-Rao and the classical Wasserstein metric. Finally, we examine the limiting behaviour of the model in terms of the parameter $\kappa$. 

\vspace{2ex}

\end{abstract} 

\section{Introduction}

Transport on networks is an important problem that arises in many areas of science, e.g. traffic on road networks \cite{bressan2014flows}, distribution of gas \cite{GasFlowNetwork2006,mindt2019entropy}, 
or transport of vesicles within neurites, \cite{hill2004fast,humpert_role_2019,tsaneva-atanasova_modeling_2007,yadaw_dynamic_2019}. On the other hand, there is the generic theory of optimal transport that describes how to move a given amount of mass at the lowest cost (see for example the books of Villani \cite{villani_topics_2003,villani2008optimal}, Santambrogio \cite{santambrogio_optimal_2015} or the survey of Brasco \cite{brasco_survey_2012} for the necessary background) and its dynamic variant introduced by Benamou and Brenier \cite{benamou_computational_2000}. 

In this paper, we aim to combine these notions by introducing a dynamic formulation of optimal transport on a network, where mass is transported along edges but can also be stored in vertices. 

In the classical theory of optimal transport the \textit{Wasserstein-distance} of order $p$ between two probability measures $\mu_1$ and $\mu_2$ on $\Omega \subset \R^n$ is defined by as
\begin{align}\label{eq:Wasserstein}
	d(\mu_1, \mu_2)^p \coloneqq \inf_{p \in \mathbb{P}(\mu_1, \mu_2)} \int_{\Omega \times \Omega} \vert x-y\vert^p  ~ dp(x,y),
\end{align}
where the symbol $\vert \cdot \vert$ denotes the Euclidean norm on $\mathbb{R}^n$ and $\mathbb{P}$ denotes the set of probability measures with marginals $\mu_1$ and $\mu_2$.

Initially introduced as a numerical algorithm, Benamou and Brenier \cite{benamou_computational_2000} introduced a dynamic version of the optimal transport problem. They showed that the calculating the Wasserstein distance is equivalent to minimising an action functional representing the kinetic energy of curves connecting the two measures $\rho_1$ and $\rho_2$, subject to a constraint given by a continuity equation, i.e.
\begin{align*}
	W_2^2(\rho_1, \rho_2)
	\coloneqq \inf_{(\mu_t, F_t )} 
	\Big\lbrace \int_0^1 \int_{\Omega} \frac{\vert F_t(x)\vert^2}{\mu_t(x)} ~ dx dt ~ \Big\vert ~ \partial_t \mu_t + \text{div}(F_t) = 0, ~ \mu_t\vert_{t=0,1} = \rho_1, \rho_2 \Big\rbrace,
\end{align*}
where $F_t$ denotes the flux and $\mu_t$ a curve in the space of probability measures. This formulation has the additional merit that is can be easily generalised, e.g. to include non-linear mobilities \cite{Lisini2010,Carrillo2010_nonlinear_mobility}. One case which is of particular interest here is when the initial and terminal measures $\rho_1$ and $\rho_2$ have different mass. Then, the classical Wasserstein distance can be replaced by the Wasserstein-Fisher-Rao metric whose dynamic formulation is given as 
\begin{align}\label{eq:FisherRao}
    \begin{split} 
    \mathcal{WFR}_\kappa^2(\rho_0, \rho_1) &= \inf_{(\gamma_t, G_t,f_t)} \Bigg \lbrace  \iint_{[0,1] \times \Omega} \frac{\vert G_t \vert^2}{2 \gamma_t}\;dxdt  + \kappa^2 \iint_{[0,1] \times \Omega} \frac{\vert f_t \vert^2}{2 \gamma_t} ~ dx dt 
    \\
    ~ &\text{s.t.} \quad \partial_t \gamma_t + \text{div } G_t  = f,  ~ \gamma_t\vert_{t=0,1} = \rho_1, \rho_2 
    \Bigg \rbrace ,
    \end{split} 
\end{align}
for $\gamma_t \in \mathcal{M}^+(\Omega)$. This allows to compute distances between measures with different masses, see \cite{chizat_interpolating_2010} for an existence result. Recently it was shown that a static version in the spirit of \eqref{eq:Wasserstein} exits, see \cite{chizat_unbalanced_2019} for details. The general theory that deals with probability measures of different mass is called \textit{unbalanced optimal transport} and was simultaneously introduced in \cite{chizat_interpolating_2010, chizat_unbalanced_2019, kondratyev_new_2016, liero_optimal_2016, liero_optimal_2018}.

More recently, \cite{monsaingeon_new_2020} introduced a new transportation model on the closure of a domain $\overline{\Omega}$ that behaves differently in the interior and on the boundary while allowing for interaction between these two. This setting can be motivated if we think of $\Omega$ as a city with a ring road $\partial \Omega$ which can only be entered upon paying a fee denoted by $\kappa$.
The overall density of cars is then $\rho = (\omega, \gamma) \in \mathcal{M}^+(\overline{\Omega}) \times \mathcal{M}^+(\partial \Omega)$, where the first entry corresponds to the cars in the inner city and the second entry to the those on the ring road. Informally, this model is given as
\begin{align}\label{eq:min_monsaignion}
    \begin{split} 
    \mathcal{W}_{\mathrm{M}}^2(\rho_0, \rho_1) &= \inf_{(\omega_t,\gamma_t,F_t,G_t,f_t)} \Bigg \lbrace \iint_{[0,1] \times \overline{\Omega}} \frac{\vert F_t \vert^2}{2 \omega_t}\;dxdt
    + \iint_{[0,1] \times \partial\Omega} \frac{\vert G_t \vert^2}{2 \gamma_t}\;dxdt 
    + \kappa^2 \iint_{[0,1] \times \partial\Omega} \frac{\vert f_t \vert^2}{2 \gamma_t} ~ d\sigma dt
    \\
    ~ &\text{s.t } ~
    \begin{matrix}
        \partial_t \omega_t + \text{div } F_t = 0  & \text{ in } \Omega,  \\
        F_t \cdot \nu = f_t & \text{ in } \partial\Omega, 
    \end{matrix}
    \quad \text{and} \quad \partial_t \gamma_t + \text{div } G_t = f_t \text{ in } \partial \Omega
    \Bigg \rbrace ,
    \end{split} 
\end{align}
where the initial concentration is defined by $\rho_0 = (\omega_0, \gamma_0)$, the terminal concentration by $\rho_1 = (\omega_1, \rho_1)$, $F_t$ denotes the momentum in $\Omega$, $G_t$ the momentum on $\partial \Omega$ and $f_t$ is the normal outflux $f_t=F_t \cdot \nu$. Existence of minimisers was shown based on duality.
Understanding the one-dimensional model as a trivial network with only one edges and two vertices (in spatial dimension one) serves as a starting point for our investigation.

In this work we introduce a dynamic formulation on a planar network. We identify edges with one-dimensional intervals on which a classical action functional is minimized while at vertices, mass may be transferred onto or off from a vertex by reaction terms as in the Fisher-Rao metric \eqref{eq:FisherRao}, rendering the transport problem on each edge similar to \eqref{eq:min_monsaignion}. After the formulation of the problem, we show that it is well-defined again using Fenchel-Rockafellar duality, and also analyse the asymptotic behaviour in terms of the  cost parameter $\kappa$.

While our model is dynamic and allows for the storage of mass at the vertices, the static $1$-Wasserstein distance on metric graphs has been studies in \cite{Mazon2015_optimal}. In their work, the authors focus on a connection between Kantorovich potentials and solutions of a $p$-Laplace problem. More recently in \cite{Erbar2021_network}, the authors consider a similar setting as ours, yet again without explicit dynamics at the vertices. They introduce a $p$-Wasserstein distance in the spirit of Benamou and Brenier and show that, as in the classical setting, absolutely continuous curves admit vector fields that solve the continuity equation. Using this characterisation of geodesics, they observe that the entropy functional is not displacement convex. However, they are still able to characterise solutions to a drift-diffusion-attraction equation as a gradient flow with respect to their distance.

This paper is organised as follows: 
In Section \ref{sec:network_setting} we provide details on the network setting, in Section \ref{sec:network_existence} we introduce our model and show existence of minimisers by means of Rockafellar duality. Section \ref{sec:Network_relationships} we discuss the relationships of the distance-functional with other metrics as the Fisher-Rao and the classical Wasserstein-metric. In the limit case, where the costs for transporting mass over the vertices diverges to infinity, we recover that our distance converges either to infinity if masses are incompatible or to the classical Wasserstein-metric if masses are compatible.
In the Appendix we present a formal calculation of the first order optimality conditions.

\section{Network setting}
\label{sec:network_setting}
We consider a planar network where edges can be identified with one-dimensional intervals. We denote the complete network by $\mathcal{G} = (\mathcal{V},\mathcal{E})$ with $\mathcal{V} = \{V^1, \ldots, V^n\}$ the set of vertices for $n \in \N$ and $\mathcal{E} = \{E^1, \ldots, E^m\}$ the set of edges for $m\in\N$.
Every vertex is defined via its coordinates in the two-dimensional space $\R^2$, i.e. $V^i \in \R^2$ for every $i \in \lbrace 1, \ldots, n \rbrace$ and every edge is homeomorphic to a one-dimensional, open interval. To each edge we assign a starting and an end point and define two functions
$ \alpha, \omega \colon \mathcal{E} \rightarrow \mathcal{V}$
that assign to every edge its starting or its end point thus determining an orientation. The functions $\bar\alpha, \, \bar\omega \colon  \{1,\ldots, m\} \to \{1,\ldots , n\}$ map a given edge the indices of the respective vertices.
By $Z(V^i)$ we denote the indices of all edges originating or ending at $V^i$ for $i \in \lbrace 1 , \ldots, n \rbrace$, i.e.
\begin{align*}
     Z(V^i) = \big\lbrace j \in \lbrace 1, \ldots, m \rbrace \colon \alpha(E^j) = V^i \vee \omega(E^j) = V^i \big\rbrace .
\end{align*}
Finally, for all $(i,j) \in \{ (i,j) \; : \; i \in \{1,\ldots, n \} \text{ and } j \in Z(V^i) \}$ we denote by $\nu_{i,j}$ the outward normal vector of edge $j$ at the point where it is connected to vertex $i$. With this notation, $\nu_{\bar\alpha(j),j}$ gives the normal at the starting point of $E_j$.
Moreover, we denote by $\mathcal{M}_+(X)$ the set of non-negative bounded measures on a given space $X$ and more precisely the set of non-negative measures on the set of edges (vertices) by
\begin{align*}
\mathcal{M}_+(\mathcal{E}) &= \mathcal{M}_+(E^1)\times \ldots \times \mathcal{M}_+(E^m),\\
\mathcal{M}_+(\mathcal{V}) &= \mathcal{M}_+(V^1)\times \ldots \times \mathcal{M}_+(V^n).
\end{align*}
Since $V^i \in \R^d$, we have that each measure in $\mathcal{M}_+(\mathcal{V})$ is of the form 
$$
\sum_{i=1}^n c_i \delta_{V_i},
$$
and therefore we identify $\mathcal{M}_+(\mathcal{V})$ with $\R_+^n$ from now on.

To formulate the dynamic optimal transport problem on the network let $\rho_0^j,\, \rho_1^j \in \mathcal{M}_+(E^j)$, $j=1,\ldots, m$ and $\gamma_0^i,\,\gamma_1^i \in \R$, $i=1,\ldots, n$ be given, and denote by $\pmb{\rho_0} = (\rho_0^1,\ldots, \rho_0^m)$, $\pmb{\rho_1} = (\rho_1^1,\ldots, \rho_1^m)$, the vector of all concentrations on edges at time $t=0$ and $t=1$ and by $\pmb{\gamma_0} = (\gamma_0^1,\ldots, \gamma_0^n), \pmb{\gamma_1} = (\gamma_1^1,\ldots, \gamma_1^n)$ the vectors of the concentration on the vertices at time $t=0$ and $t=1$. Next, on the closed set $\Omega_{\mathcal{G}} = \bigcup_{i=1}^n V^i \cup  \bigcup_{j=1}^m E^j$ we define the measure that translates to the total density on the network by
\begin{align}\label{eq:total_mass_network}
\varsigma_l = \sum_{j=1}^m\rho_l^j + \sum_{i=1}^n \gamma_l^i, \quad l \in \{0,1\},
\end{align}
and make the assumption that our initial and final data $(\pmb{\rho_0},\pmb{\rho_1},\pmb{\gamma_0},\pmb{\gamma_1})$ are such that $\varsigma_0,\, \varsigma_1 \in \mathcal{P}(\Omega_{\mathcal{G}})$ holds.

On every edge and every vertex, i.e. for every $j \in \lbrace 1, \ldots, m \rbrace$ and $i\in \lbrace 1, \ldots , n \rbrace$, we consider the following \textit{continuity equation} on the network $\mathcal{G}$
\begin{align}
    \label{eq:ContuinityEquation}
    \partial_t \rho^j_t + \partial_x F^j_t = 0 \text{ in } E^j, \; \quad  \partial_t \gamma^i_t = f^i_t \text{ on } V^i \quad \text{ with } f^i_t = \sum_{j \in Z(V^i)} F^j_t(V_i) \cdot \nu_{i,j}, 
\end{align}
where $F^j_t \colon E^j \times (0,T] \rightarrow \R , f^i_t \in \R$, the space derivative $\partial_x F^j_t$ is calculated with respect to the orientation of the edge. 
We will give a rigorous definition of weak solution in the next section. A sketch of this situation is shown in Figure \ref{fig:network}. Note that as $\mathcal{V}$ is discrete, the time-dependent measures $\gamma^i_t, f^i_t$ for $i \in \lbrace 1 , \ldots ,  n \rbrace$ are given by
\begin{align*}
    \gamma^i_t = C_1(i,t) \delta_{V^i}, \qquad f^i_t = C_2(i,t) \delta_{V^i}
\end{align*}
where  $\delta_{V^i}$ denotes the Dirac-measure at the vertex $V^i$ and $C_1(i,t), C_2(i,t) \in \R^+$ are time-dependent constants. Thus, we identify the measures with their respective time-dependent constants while, by abuse of notation, still writing  $\gamma^i_t$ and $f^i_t$ in the following.


For a given network concentration $(\pmb{\rho_0},\pmb{\rho_1},\pmb{\gamma_0},\pmb{\gamma_1})\in \mathcal{M}_+(\mathcal{E})\times \mathcal{M}_+(\mathcal{E}) \times \mathcal{M}_+(\mathcal{V}) \times \mathcal{M}_+(\mathcal{V})$ with $\varsigma_0,\, \varsigma_1 \in \mathcal{P}(\Omega_{\mathcal{G}})$, we consider the \textit{minimisation-problem} of an action being the combination of  Wasserstein and the Fisher-Rao terms
\begin{align}
    \label{eq:minimizaionprob}
    \min_{(\pmb{\rho_t},\pmb{F_t},\pmb{\gamma_t},\pmb{f_t}) \in \mathcal{CE} (\pmb{\rho_0},\pmb{\rho_1},\pmb{\gamma_0},\pmb{\gamma_1})} \left\{ \sum_{j=1}^m \iint_{\overline{E^j} \times [0,1]} \frac{|F^j_t|^2}{2\rho^j_t}\;dxdt + \kappa^2\sum_{i=1}^n  \int_{[0,1] } \frac{|f^i_t|^2}{2\gamma^i_t}\; dt \right\},
\end{align}
where $\kappa > 0$ is a given constant and with
\begin{align*}
    \mathcal{CE}(\pmb{\rho_0},\pmb{\rho_1},\pmb{\gamma_0},\pmb{\gamma_1}) &= \Big\lbrace (\pmb{\rho_t},\pmb{F_t},\pmb{\gamma_t},\pmb{f_t}) \text{ that fulfil \eqref{eq:ContuinityEquation}} \text{ and } \pmb{\rho_{t=0}} = \pmb{\rho_0}, \pmb{\rho_{t=1}} = \pmb{\rho_1}, \, \pmb{\gamma_{t=0}} = \pmb{\gamma_0}, \pmb{\gamma_{t=1}} = \pmb{\gamma_1} \Big\rbrace.
\end{align*}
In this setting - at least on a formal level - the total mass on the network is conserved as no mass may enter or leave the system, i.e.
\begin{align*}
   \frac{d}{dt}\text{Mass}(\mathcal{G}, t) &= \frac{d}{dt}\left[ \sum_{j=1}^m \int_{\overline{E^j}} \rho^j_t ~ dx + \sum_{i=1}^n \gamma^i_t\right] = 0
\end{align*}
for every $t \in [0,T]$.
Moreover, for brevity we write $\Vert \rho_t^j \Vert = \Vert \rho_t^j \Vert_{L^1(E^j)}$ in the whole paper.

\begin{remark}[Generalisations]
    \begin{enumerate}
        \item[a)] An interesting generalisation of the current setting is the case where we allow mass in- and outflow at the vertices of grad 1, i.e. the outer vertices that are only connected with one edge.
        \item[b)] The generalisation to a non-connected network with finite many connected components is possible but we omit the proof for readability. 
    \end{enumerate}
    
\end{remark}

\begin{remark}[Kantorovich formulation and limit problem]
Another interesting question for further research is a static formulation of both \eqref{eq:min_monsaignion} and \eqref{eq:minimizaionprob}. For the first one, based on the explicit calculation of geodesics in \cite{monsaingeon_new_2020} between two point masses, one being located within the domain and one at the boundary, we conjecture that \eqref{eq:min_monsaignion} allows for a Kantorovich formulation with cost 
\begin{align}\label{eq:cost_limit}
c(x,y) = \begin{cases} 
            \frac{1}{2}|x-y|^2 & x,\, y \in \Omega\text{ or } x,\,y \in \partial \Omega,\\
            \frac{1}{2}|x-y|^2 + \kappa\sqrt{|x-y|^2+\kappa^2} + \kappa^2 & x\in \Omega, \, y \in \partial\Omega \text{ or } x\in \partial \Omega,\, y \in \Omega.
            \end{cases}
\end{align}
Furthermore, one might ask if one starts with a standard optimal transport problem on domain $\Omega = I_1 \cup I_2$ and let the size of $I_2$ go to zero, is there an appropriate rescaling of the cost so that we obtain \eqref{eq:min_monsaignion} in the limit? This could be examined either on the level of optimality conditions or in the respective static solutions where the limit cost would then need to be \eqref{eq:cost_limit}.
\end{remark}

\begin{figure}
	\begin{center}
		\begin{tabular}{cc}
			\includegraphics[width=0.55 \textwidth]{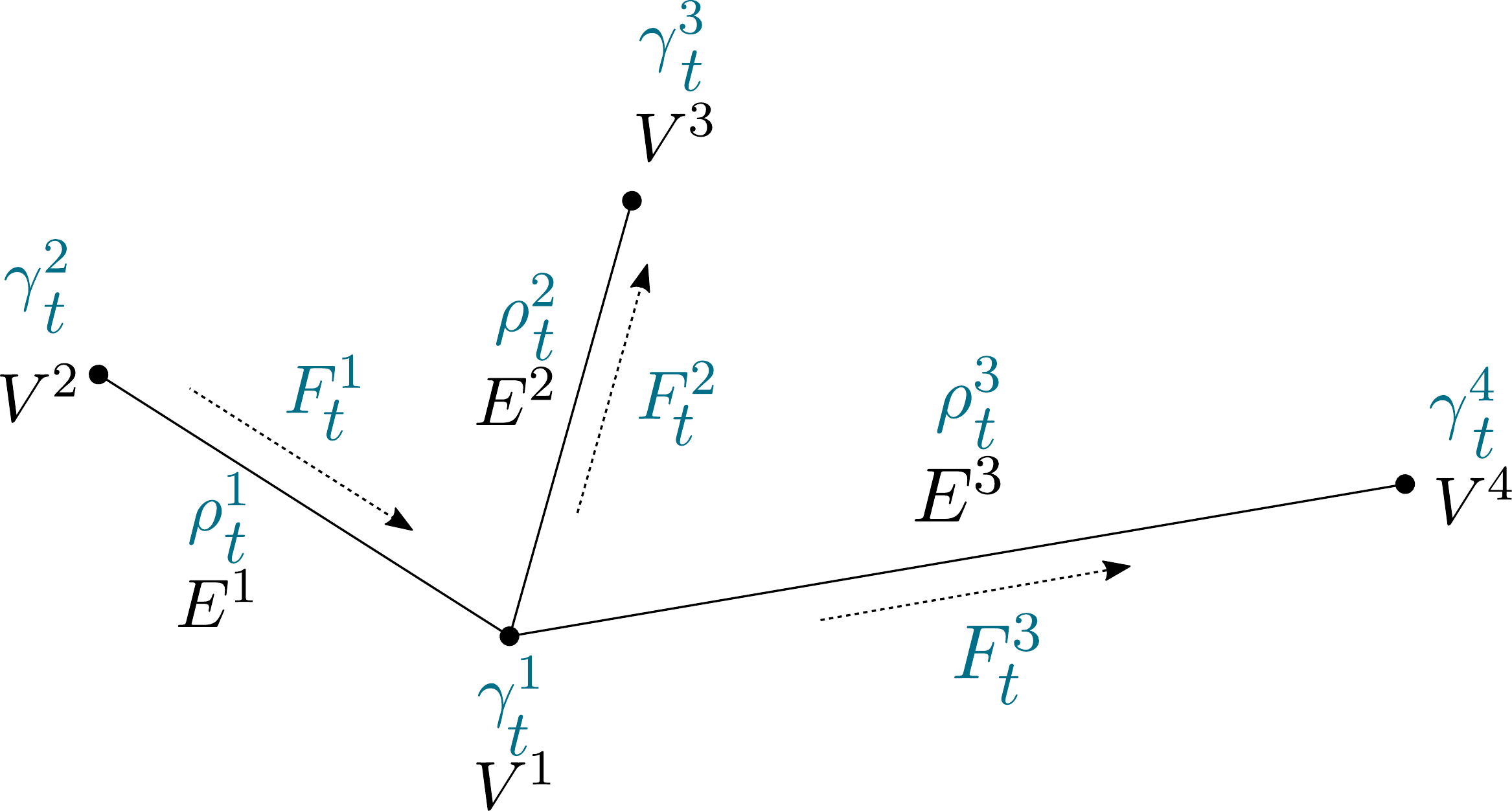} &
		\end{tabular}
	\end{center}
	\caption{Example of a network configuration with three edges $\mathcal{E}= \lbrace E^1, E^2, E^3 \rbrace$ and four vertices $\mathcal{V}= \lbrace V^1, V^2, V^3, V^4 \rbrace$. The concentrations on $E^j$ are given by $\rho^j_t$ for $j \in \lbrace 1, 2, 3 \rbrace$ and the concentration on the vertices is given by $\gamma^i_t$ on $V^i$ with $i \in \lbrace 1,2,3, 4 \rbrace$. Labels of edges and vertices are shown in black and concentrations on those are shown in blue.
	The index set corresponding to the edges originating or ending at $V^1$ is given by $Z(V^1) = \lbrace 1,2, 3\rbrace$ as $\omega(E^1)= \alpha(E^2) = \alpha(E^3) = V^1$.}
	\label{fig:network}
\end{figure}

\section{Transport model and existence of minimisers}
\label{sec:network_existence}

In this section, we show existence of minimisers of \eqref{eq:minimizaionprob}, based on a duality argument and by extending the strategy of  \cite{monsaingeon_new_2020} to the network setting.

\subsection{Continuity equations and action functional}
We denote by $Q_{\mathcal{G}} =  [0,1] \times \Omega_{\mathcal{G}}$ the space-time domain of the whole network, 
\begin{align*}
    Q_{\overline{E^j}} \coloneqq [0,1] \times \overline{E^j}, \quad \text{ and } \quad Q_{V^i} \coloneqq [0,1] \times V^i.
\end{align*}
The precise notion of weak solution for the continuity equation \eqref{eq:ContuinityEquation} is given as follows.
\begin{definition}[weak solution]
    \label{def:weakformcontinuityequation}
    Given $(\pmb{\rho_0},\pmb{\rho_1},\pmb{\gamma_0},\pmb{\gamma_1})\in \mathcal{M}_+(\mathcal{E})\times \mathcal{M}_+(\mathcal{E}) \times \mathcal{M}_+(\mathcal{V}) \times \mathcal{M}_+(\mathcal{V})$, we denote by $\mathcal{CE}(\pmb{\rho_0},\pmb{\rho_1},\pmb{\gamma_0},\pmb{\gamma_1})$ the set of measures $(\pmb{\rho_t},\pmb{F_t},\pmb{\gamma_t},\pmb{f_t})$ which satisfy the continuity equation \eqref{eq:ContuinityEquation} in the following weak sense
\begin{align}
    \begin{split}
    \label{eq:weakformualtionedges}
    \sum_{j=1}^m \Big[\iint_{ Q_{\overline{E^j}}} \partial_t \varphi_t \;d\rho^j_t &+ \partial_x \varphi_t \;dF^j_t \Big] - \sum_{i=1}^n \phi_t(V_i)f_t^i
    \\
    &= \sum_{j=1}^m \Big[\int_{\overline{E^j}} \varphi_1 \;d\rho^j_1 - \int_{\overline{E^j}} \varphi_0\;d\rho^j_0 \Big]
    \end{split}
\end{align}
for all test functions $\varphi_t \in C^1(Q_{\mathcal{G}})$ and 
\begin{align}
    \label{eq:weakformualtionvertices}
    \sum_{i=1}^n \Big[\int_{[0,1]} \partial_t \varphi^i_t \; \gamma^i_t dt \Big]
    =
    \sum_{i=1}^n \Big[\int_{[0,1] } \varphi^i_t \;df^i_t \Big]
    +
    \sum_{i=1}^n \Big[ \varphi^i_1 \gamma^i_1 - \varphi^i_0 \gamma^i_0 \Big]
\end{align}
for all $\varphi^i \in C^1([0,1])$, $i=1,\ldots, n$ denoting the test function corresponding to the vertex $V^i$.
\end{definition}
\begin{remark}
Note that as we are integrating over $\overline{E^j}$, formally each continuity equation produces a boundary term $\phi_t(\omega(E_j)) F_j\cdot \nu_{\bar\alpha(j),j} + \phi_t(\alpha(E_j)) F_j\cdot \nu_{\bar\omega(j),j}$. Summing over all equations this yields
\begin{align*}
&\sum_{j=1}^m(\phi_t(\omega(E_j)) F_t^j(\omega(E_j))\cdot \nu_{\bar\omega(j),j} + \phi_t(\alpha(E_j)) F_t^j(\alpha(E_j))\cdot \nu_{\bar\alpha(j),j})\\
&= \sum_{i=1}^n \phi_t(V_i)\sum_{j \in Z(V_i)}  F_t^j(V_i)\cdot \nu_{i,j} = \sum_{i=1}^{n} \phi_t(V_i) f_t^i,
\end{align*}
using the coupling of fluxes in \eqref{eq:ContuinityEquation} which is therefore incorporated in the weak formulation.
\end{remark}

The continuity equation \eqref{eq:ContuinityEquation} can also be considered in the global sense, i.e. can be formulated as an equation on the whole network. This results in:

\begin{proposition}(global continuity equation)
    \label{prop:globalequation}
    Let $(\pmb{\rho_t},\pmb{F_t},\pmb{\gamma_t},\pmb{f_t}) \in \mathcal{CE}(\pmb{\rho_0},\pmb{\rho_1},\pmb{\gamma_0},\pmb{\gamma_1})$ and
    define by $\tilde{F}^j_t$ the trivial extension of $F^j_t$ on the whole network, i.e. $\tilde{F}^j_t = 0$ on all other edges $E^k$ with $k \neq j$, (the measures $\tilde{\gamma}^j_t, \tilde{f}^i_t$ and $\tilde{\rho}^j_t$ are defined analogue).
    We define the following global variables
    \begin{align*}
            \varsigma_t = \sum_{j=1}^m\tilde{\rho}^j_t + \sum_{i=1}^n \tilde{\gamma}^i_t, \quad  H_t = \sum_{j=1}^m \tilde{F}^j_t \quad  \text{ and } \quad h_t = \sum_{i=1}^n \tilde{f}^i_t.
    \end{align*}
    Then, the existence of the global continuity equation
    \begin{align*}
        \begin{cases} \partial_t \varsigma_t + \partial_x H_t = h_t & \text{in } \mathcal{Q}_\mathcal{G},  \\
        \sum_{j \in Z(V^i)} F^j_t \cdot \nu_{i,j} = f^i_t & \text{in } \mathcal{Q}_\mathcal{V}. \end{cases}
    \end{align*}
    in the weak sense with initial/terminal data $\pmb{\rho_0}, \pmb{\rho_1}, \pmb{\gamma_0}, \pmb{\gamma_1}$ is a consequence of the existence of \ref{def:weakformcontinuityequation}.
\end{proposition}

\begin{proof}
     Choosing $\varphi^i_t = \varphi|_{V^i}$ in \eqref{eq:weakformualtionvertices} (where the pointwise restriction exists as $\varphi \in C^1(Q_{\mathcal{G}})$) and summing up \eqref{eq:weakformualtionedges} and \eqref{eq:weakformualtionvertices}, we recover the claimed weak formulation
   \begin{align*}
       \iint_{\mathcal{Q}_\mathcal{G}} &\partial_t \varphi_t \; d \Big(\sum_{j=1}^m\tilde{\rho}^j_t + \sum_{i=1}^n \tilde{\gamma}^i_t \Big)  
       + \iint_{\mathcal{Q}_\mathcal{G}} \partial_x \varphi_t \; d \Big(\sum_{j=1}^m \tilde{F}^j_t \Big) - \iint_{\mathcal{V} \times [0,1]} \varphi_t \;d h_t
       \\
       &\quad =  \iint_{\mathcal{Q}_\mathcal{G}} \varphi_t \; d \Big(\sum_{i=1}^n \tilde{f}^i_t \Big)
       + \int_\mathcal{G} \varphi_1 \;  d \Big( \sum_{j=1}^m\tilde{\rho}^j_1 + \sum_{i=1}^n \tilde{\gamma}^i_1 \Big) 
       -  \int_\mathcal{G} \varphi_0 \; d\Big( \sum_{j=1}^m\tilde{\rho}^j_0 + \sum_{i=1}^n \tilde{\gamma}^i_0 \Big) 
   \end{align*}
   for all $\varphi_t \in C^1(\mathcal{Q}_\mathcal{G})$.
\end{proof}


To rigorously define the minimisation of the action functional, we need the following definition.
\begin{definition}[generalised Lagrangian] 
    For $(a, b) \in \R^+ \times \R $  we define the action density
    \begin{align*}
        A(a, b) = \left\{\begin{array}{cl}
        \frac{|b|^2}{2a} & \text{ if } a > 0,\\
        0 & \text{ if } (a, b) = (0,0),\\
        + \infty & \text{ otherwise.} \\
        \end{array}\right.
    \end{align*}
For $\pmb{a} \in \R^n$, $\pmb{b} \in \R^n$  and $\pmb{\eta} = ( \pmb{\eta}^1, \ldots, \pmb{\eta}^n)\in \R^{|Z(V^1)|} \times \cdots \times \R^{|Z(V^n)|}$ with  $\pmb{\eta}^{i} = (\eta^{i,j})_{j \in Z(V^i)}$, 
we also introduce the extended action density as 
\begin{align}
\bar A (\pmb{a}, \pmb{b}; \pmb{\eta}) = \sum_{i=1}^n \left(A(a^i,b^i) + \begin{cases} 0 & \text{if } a^i +  \sum_{k\in  Z(V^i)} \eta^{i,k} = 0, \\ +\infty & \text{otherwise}. \end{cases}\right)
\end{align}
\end{definition}


Let us remark that $A(a, b)$ is convex in both variables, lower semi-continuous and 1-homogeneous. This allows us to give a rigorous definition of the action functional. 
\begin{definition}[action functional]
For $\pmb{\mu_t} = (\pmb{\rho_t}, \pmb{F_t}; \pmb{\gamma_t}, \pmb{f_t})$, we define the action functional as
\begin{align}
    \begin{split}
    \label{eq:actionfunctional}
    \mathcal{A}(\pmb{\mu_t}) &\coloneqq  \sum_{j=1}^m \iint_{ Q_{\overline{E^j}}} A\Big(\frac{d \rho^j_t}{d\theta^j}, \frac{d F^j_t}{d\theta^j} \Big) \; d\theta^j + \kappa^2 \sum_{i=1}^n  \int_0^1 A(\gamma^i_t, f^i_t ) \; dt,
    \end{split}
\end{align}
where $\theta^j$ are non-negative Borel reference measures such that $\vert \rho^j_t \vert \ll \theta^j,  \vert F^j_t \vert \ll \theta^j$.
Since $A$ is jointly 1-homogeneous, this definition does not depend on the choice of $\theta^j_\rho$.
\end{definition}
Note that this action functional is lower semi-continuous, see \cite[Theorem 3.3]{bouchitte_new_1990}. We now define the bounded Lipschitz-distance between two measures $\rho^j_0, \rho^j_1 \in \mathcal{M}^+(E^j)$ on an edge and, respectively, for two measures  $\gamma^i_0, \gamma^i_1 \in \mathcal{M}^+(V^i)$ on a vertex, as
\begin{align*}
    d_{\text{BL}, \mathcal{E}} (\rho^j_0, \rho^j_1) &= \sup \Big\lbrace \Big\vert \int_{\overline{E^j}} \Phi d(\rho^j_1 - \rho^j_0) \Big\vert \quad \text{s.t.} \quad \Vert \Phi \Vert_\infty + \text{Lip}(\Phi) \leq 1 \Big\rbrace,
    \\
    d_{\text{BL}, \mathcal{V}} (\gamma_0^i, \gamma_1^i) &=  \vert \gamma_0^i - \gamma_1^i \vert,
\end{align*}
where $\Phi \colon E^j \rightarrow \R$ is a Lipschitz-continuous function with $\text{Lip}(\Phi)$ denoting its Lipschitz-constant (again, by abuse of notion we denote by $\gamma^i_t$ a measure as well as the corresponding constant).
It is well known that the bounded Lipschitz-distance metrises the narrow convergence of probability measures.
Let us continue by summarising some properties of solutions to the continuity equation.
\begin{proposition}[properties of solutions of continuity equations]
    \label{prop:propertiesofthesolutions}
    Any admissible quadruple 
    $(\pmb{\rho_t}, \pmb{F_t}; \pmb{\gamma_t}, \pmb{f_t}) \in \mathcal{CE}(\pmb{\rho_0},\pmb{\rho_1},\pmb{\gamma_0},\pmb{\gamma_1})$ can be disintegrated in time as
    \begin{align*}
        d \rho^j_t(x,t) = d \rho^j_t(x) dt \qquad 
    \end{align*}
    for all $j \in \lbrace 1,  \ldots , m \rbrace$.
    If the action functional $\mathcal{A}(\pmb{\rho_t}, \pmb{F_t}, \pmb{\gamma_t}, \pmb{f_t})< \infty$ is finite, then we obtain the following:
    \begin{enumerate}
        \item[(i)] For every $t$, the measures $\rho^j_t, \gamma^i_t$ are non-negative for all $ i \in \lbrace 1, \ldots, n\rbrace, j \in \lbrace 1, \ldots , m \rbrace$.
        Moreover, we the total mass $\varsigma_t$ defined in \eqref{eq:total_mass_network} is conserved, i.e.
        \begin{align*}
            \Vert \varsigma_t \Vert = \sum_{j=1}^m \Vert \rho^j_t \Vert + \sum_{i=1}^n  \gamma^i_t = 1, \qquad \text{ for almost every } t \in [0,1].
        \end{align*}
        \item[(ii)] For every $ j \in \lbrace 1, \ldots, m \rbrace$ the Radon-Nikodym densities
        \begin{align*}
            u^j_t (x) \coloneqq \frac{d F^j_t}{d\rho^j_t}, 
        \end{align*}
        are well-defined $d\rho^j_t$ almost everywhere and we obtain an alternative formulation of $\mathcal{A}(\pmb{\mu_t})$ as
        \begin{align*}
            \mathcal{A}(\pmb{\rho_t}, \pmb{F_t}; \pmb{\gamma_t}, \pmb{f_t}) &= \frac{1}{2} \sum_{j=1}^m \iint_{\mathcal{Q}_{\overline{E^j}}} \vert u^j_t \vert^2 d\rho^j_t +  \frac{1}{2} \sum_{i=1}^n \int_0^1 \kappa^2 \frac{\vert f^i_t \vert^2}{\gamma^i_t} ~ dt 
            \\
            &= \frac{1}{2} \sum_{j=1}^m \iint_{\mathcal{Q}_{\overline{E^j}}} \vert u^j_t \vert^2 d \rho^j_t dt
            + \frac{1}{2} \sum_{i=1}^n \int_0^1 \kappa^2 \frac{\vert f^i_t \vert^2}{\gamma^i_t} ~ dt  \;.
        \end{align*}
        \item[(iii)] The curves $ t\mapsto \rho^j_t \in \mathcal{M}(\overline{E^j})$ and $ t\mapsto \gamma^i_t \in \mathcal{M}(V^i)$ are narrowly continuous for all $ i \in \lbrace 1, \ldots, n\rbrace, j \in \lbrace 1, \ldots, m \rbrace$ and satisfy the bounded-Lipschitz estimate
        \begin{align}
            \label{eq:Abschaetzungboundedlipschitzdistance}
            \sum_{j=1}^m d_{\text{BL}, \overline{E^j}} (\rho^j_s, \rho^j_t ) 
            + \sum_{i=1}^n d_{\text{BL}, V^i}(\gamma^i_s, \gamma^i_t) \leq C_\kappa \sqrt{\mathcal{A}((\pmb{\rho_t}, \pmb{F_t}; \pmb{\gamma_t}, \pmb{f_t}))} \vert t-s\vert^{\frac{1}{2}}
        \end{align}
        for every time $s,t \in [0,1]$ and the constant $C_\kappa = 2\sqrt{2(n+m)} \max\lbrace 1,\frac{1}{\kappa} \rbrace$.
        In particular the initial/terminal conditions are taken in the narrow sense.
    \end{enumerate}
\end{proposition}

\begin{proof}
    For (i) and (ii) we refer to \cite[Prop. 3.5]{monsaingeon_new_2020}. For (iii) note that the bounded Lipschitz-distance metrises the narrow convergence of measures, thus it suffices to establish equation \eqref{eq:Abschaetzungboundedlipschitzdistance}.
        We only sketch the proof for $t \mapsto \rho_t^j$, as the argument is similar for the curve $t \mapsto \gamma_t^i$.
        For fixed $\Phi \in C^1(\mathcal{G})$ we will estimate from below the time derivatives of 
        \begin{align*}
            l^j_t \coloneqq \int_{\overline{E^j}} \Phi(x) ~ d\rho^j_t(x).
        \end{align*}
        Using the weak form of the continuity equations and arguing as in \cite[Prop. 3.5]{monsaingeon_new_2020} we obtain
        \begin{align*}
            \sum_{j=1}^m \Big\vert \frac{d l^j_t}{dt} \Big\vert 
            \leq \Vert \partial_x \Phi \Vert_\infty \sum_{j=1}^m \int_{\overline{E^j}} \vert u_t^j \vert ~ d \rho^j_t + \Vert \Phi \Vert_\infty  \sum_{i=1}^n \vert f^i_t \vert \; .
        \end{align*}
        Next, we use the facts that
        \begin{itemize}
            \item \(
            \begin{aligned}[t]
                \Vert u^j_t \Vert_{L^1(\overline{E}_j;d\rho_t^j)} \leq \Vert u^j_t \Vert_{L^2(\overline{E}_j;d\rho_t^j)} \text{ as well as } \Vert \rho^j_t \Vert \leq 1 \text{ and } \gamma^i_t \leq 1 \text{ (using (i))},
               \\
            \end{aligned}
          \)
            \item \(
            \begin{aligned}[t]
               \sum_{i=1}^n \sqrt{a_i} \leq \sqrt{n} \Big( \sum_{i=1}^n a_i \Big)^\frac{1}{2} \text{ for all $a_i \geq 0$, by H\"older's inequality}, \\
            \end{aligned}
          \)
        \end{itemize}
        to calculate
        \begin{align*}
            \sum_{j=1}^m  \Big\vert \frac{d l^j_t}{dt} \Big\vert 
            &\leq \big( \Vert \partial_x \Phi \Vert_\infty + \Vert \Phi \Vert_\infty  \big) \Big(\sum_{j=1}^m \int_{\overline{E^j}} \vert u_t^j \vert~ d \rho^j_t +  \sum_{i=1}^n \vert f^i_t \vert  \Big)
            \\
            &\leq \sqrt{n+m} \big( \Vert \partial_x \Phi \Vert_\infty + \Vert \Phi \Vert_\infty  \big) \Big(\sum_{j=1}^m \int_{\overline{E^j}} \vert u_t^j \vert^2~ d \rho^j_t +  \sum_{i=1}^n  \vert f^i_t \vert ^2  \Big)^\frac{1}{2} \in L^2(0,1)  .
        \end{align*}
        Thus $l$ is absolutely continuous, and we obtain
        \begin{align*}
            & \sum_{j=1}^m \Big\vert \int_{\overline{E^j}} \Phi ~d(\rho^j_t - \rho^j_s ) \Big\vert 
            = \sum_{j=1}^m \vert l^j_t - l^j_s \vert 
           \leq \sum_{j=1}^m \int_s^t \Big\vert \frac{d l^j_\tau}{d\tau} \Big\vert ~d\tau 
            \leq \sum_{j=1}^m \Big\Vert \frac{d l^j_t}{d\tau} \Big\Vert_{L^2(0,1)} \vert t- s\vert^\frac{1}{2}
            \\
            &\leq \sqrt{n+m} \big( \Vert \partial_x \Phi \Vert_\infty + \Vert \Phi \Vert_\infty  \big) \Big(\sum_{j=1}^m \iint_{[0,1] \times \overline{E^j}} \vert u_t^j \vert^2~ d \rho^j_t dt +  \sum_{i=1}^n \int_0^1 \vert f^i_t \vert^2 ~ dt \Big)^\frac{1}{2}  \vert t- s\vert^\frac{1}{2} 
            \\
            &\leq \sqrt{2(n+m)} \big(\Vert \partial_x \Phi \Vert_\infty + \Vert \Phi \Vert_\infty  \big) \Big(\sum_{j=1}^m \iint_{[0,1] \times \overline{E^j}} \frac{\vert u_t^j \vert^2}{2}~ d \rho^j_t dt 
            \\
            &\hspace{20ex} + \frac{1}{\kappa^2}  \sum_{i=1}^n  \int_0^1 \kappa^2 \frac{ \vert f^i_t \vert^2}{2 \gamma^i} ~ dt \Big)^\frac{1}{2}  \vert t- s\vert^\frac{1}{2} 
            \\
            &\leq \sqrt{2(n+m)} \big(\Vert \partial_x \Phi \Vert_\infty + \Vert \Phi \Vert_\infty  \big) \max\Big\lbrace 1, \frac{1}{\kappa} \Big\rbrace \sqrt{\mathcal{A}(\pmb{\mu_t})} \vert t- s\vert^\frac{1}{2} \; .
        \end{align*}
    Taking the supremum over all $\Phi$ with $\Vert  \Phi \Vert_\infty + \operatorname{Lip}(\Phi) \leq 1$, we recover the desired estimate for the bounded Lipschitz distance. Analogously we obtain the same inequality for $\gamma^i_t$, and thus the constant in the proposition is given by 
    \begin{align*}
        C_\kappa = 2\sqrt{2(n+m)} \max\Big\lbrace 1, \frac{1}{\kappa}\Big\rbrace \; . 
    \end{align*}
\end{proof}

\subsection{Duality and Existence}
\begin{definition}
    \label{def:quantityW}
    For any admissible network concentration $(\pmb{\rho_0}, \pmb{\rho_1}, \pmb{\gamma_0}, \pmb{\gamma_1})$ and $\pmb{\mu_t} \in \mathcal{CE}(\pmb{\rho_0}, \pmb{\rho_1}, \pmb{\gamma_0}, \pmb{\gamma_1})$, we define the quantity
    \begin{align}
        \label{eq:distancefunctional}
        \mathcal{W}_\kappa^2(\pmb{\rho_0},\pmb{\rho_1},\pmb{\gamma_0},\pmb{\gamma_1}) \coloneqq \inf_{\pmb{\mu_t} \in \mathcal{CE}(\pmb{\rho_0},\pmb{\rho_1},\pmb{\gamma_0},\pmb{\gamma_1})}  \mathcal{A}(\pmb{\rho_t}, \pmb{F_t}; \pmb{\gamma_t}, \pmb{f_t}).
    \end{align}
\end{definition}

It will turn out in Proposition \ref{prop:quantityisdistance} that this quantity is even a distance but first of all we have to show that it is always well-defined.
\begin{lemma}
    \label{lem:distancefinite}
    For every $(\pmb{\rho_0},\pmb{\rho_1},\pmb{\gamma_0},\pmb{\gamma_1}) \in \mathcal{M}_+(\mathcal{E}) \times \mathcal{M}_+(\mathcal{E}) \times \mathcal{M}_+(\mathcal{V}) \times \mathcal{M}_+(\mathcal{V})$ such that $\varsigma_0,\, \varsigma_1 \in \mathcal{P}(\Omega_{\mathcal{G}})$, the quantity $\mathcal{W}_\kappa^2(\pmb{\rho_0},\pmb{\rho_1},\pmb{\gamma_0},\pmb{\gamma_1})$ is finite.
\end{lemma}

\begin{proof}
    Let $V^k \in \mathcal{V}$ be an arbitrary vertex. We will  show that any element $(\pmb{\rho_0, \gamma_0})$ can be connected to $(\pmb{0}, \pmb{\delta_{V^k}})$ with finite cost,
    where $\pmb{\delta_{V^k}} = (0,\ldots, \delta_{V^k}, \ldots, 0) \in \R^n$ denotes the concentration on the vertices with $\delta$ being the Dirac-measure.
    By symmetry $(\pmb{\rho_0}, \pmb{\delta_{V^k}})$ can also be connected to any other $(\pmb{\rho_1}, \pmb{\gamma_1})$, thus connecting $\pmb{\rho_0}$ with $\pmb{\rho_1}$ with finite cost.
    As proven in \cite[Lemma 3.7]{monsaingeon_new_2020}, it is enough to show that this can be done in a finite number of elementary steps with finite cost combined with a time re-scaling argument.
    As we assumed that our network is finite and connected, the assertion following by from finitely many consecutive applications of \cite[Lemma 3.7]{monsaingeon_new_2020}.
\end{proof}

Closely following to \cite{chizat_unbalanced_2019}, we now proof existence of minimisers of \eqref{eq:distancefunctional} using a duality argument and the Fenchel-Rockafellar theorem, see \cite{rockafellar_duality_1967} or \cite[Thm. 31.1]{rockafellar_convex_1970}.

\begin{theorem}[Fenchel-Rockafellar theorem]
    \label{th:Rockafellar}
    Let $X_1, X_2$ be normed vector spaces with topological duals $X_1^*, X_2^*$ and $L \colon X_1 \rightarrow X_2$ be a bounded linear operator with adjoint $L^*:X_2^* \to X_1^*$. Furthermore, let $\mathcal{F}\colon X_1 \rightarrow \R \cup \lbrace - \infty \rbrace$ and $\mathcal{G}\colon X_2 \rightarrow \R \cup \lbrace - \infty \rbrace$ be two proper, concave functions. 
    If there exists $x \in X_1$ such that $\mathcal{F}(x)$ is finite and $\mathcal{G}$ is continuous at $y = Lx$, then
    \begin{align*}
        \sup_{x\in X_1} \lbrace \mathcal{F}(x) + \mathcal{G}(Lx) \rbrace
        = \min_{y^* \in X_2^*} \lbrace - \mathcal{F}^*(L^*y^*) - \mathcal{G}^*(y^*) \rbrace,
    \end{align*}
    where $\mathcal{F}^*$ denotes the Fenchel-Legendre conjugate of $\mathcal{F}$, respectively $\mathcal{G}$.
    Moreover, if there exists $y^* \in X_2^*, x \in X_1$ such that $L^*y^* \in \partial(-\mathcal{F})(x)$ and $Lx \in \partial (- \mathcal{G}^*)(y^*)$ then $x$ achieves the $\sup$ and $y^*$ is a minimizer.
\end{theorem}

Next, we define two subsolution sets, one corresponding to the edges and one to the vertices,
\begin{align}
    \begin{split}
    \label{eq:DefinitionS}
    S_E &\coloneqq \Big\lbrace (\alpha, \beta ) \in \R^+ \times \R \colon \qquad \alpha + \frac{\vert \beta \vert^2}{2} \leq 0 \Big\rbrace,
    \\
    S_{V^i}^\kappa &\coloneqq \Big\lbrace (a,b,\pmb{c}) \in \R^+ \times \R \times \R^{\vert Z(V^i) \vert} \colon \quad a + \frac{\vert b - \frac{1}{|Z(V^i)|}\sum_{j \in Z(V^i)} c_j\vert^2}{2\kappa^2} \leq 0 \Big\rbrace
    \end{split}
\end{align}
and the convex indicator functions of these sets
\begin{align*}
    \iota_{S_E} (\alpha, \beta ) \coloneqq \begin{cases} 0 &\text{if } (\alpha, \beta) \in S_E, \\ +\infty &\text{otherwise}, \end{cases}
    \quad 
    \text{and}
    \quad
    \iota_{S_{V^i}^\kappa} (a,b,\pmb{c}) \coloneqq \begin{cases} 0 &\text{if } (a,b,\pmb{c}) \in S_{V^i}^\kappa, \\ +\infty &\text{otherwise}. \end{cases}
\end{align*}
For given $i,j$, the variables $\alpha, \beta$ will be dual multipliers to $\rho^j_t, F^j_t$ and $a$ will be dual to $\gamma^i_t$. The variable $b$ will be dual to $f^i_t$ and $\pmb {c}$ will be dual to $\sum_{j \in Z(V^i)} F^j_t \cdot \nu_{i,j}$.
Their sum $b - \sum c$ translates to the fact that no mass gets lost at a vertex.
In the following we will take $(\alpha, \beta ) = ( \partial_t \phi^j_t, \partial_x \phi^i_t$) and $(a,b,c) = (\partial_t \psi^i_t, \psi, \phi \vert_{\mathcal{V}})$ for suitable test functions $\phi^j_t \in C^1, \psi^i_t \in C^1$. Then, $(\partial_t \phi^j_t, \partial_x \phi^i_t) \in S_E$ and $\partial_t \psi^i_t, \psi, \phi \vert_{\mathcal{V}} \in S_{V^i}^\kappa$ mean that $\phi^j_t, \psi^i_t$ are (smooth) sub-solutions of the Hamilton-Jacobi system
\begin{align}
    \label{eq:Lagrangedual}
    \partial_t \phi^j_t + \frac{1}{2} \vert \partial_x \phi^j_t \vert ^2 \leq 0 
    \quad \text{and} \quad
    \partial_t \psi^i_t + \frac{1}{2 \kappa^2} \Big\vert \psi^i_t - \frac{1}{|Z(V^i)|}\sum_{j\in Z(V^i)} \phi^j_t \Big\vert^2 \leq 0,
\end{align}
(see Appendix \ref{sec:derivationHJ} for the derivation of the Hamilton-Jacobi equations by formal Lagrangian calculus).
As in Monsaingeon \cite{monsaingeon_new_2020}, this system of coupled Hamilton-Jacobi equations is invariant under the addition of a common constant to $\psi_t^i$ and $\phi_t^j$, $j \in Z(V^i)$. 
Consequently, the convex closed set $S_{V^i}^\kappa$ is thus invariant under diagonal shifts $b+k , \pmb{c} +k$, for every constant $k \in \R$.

\begin{lemma} 
    \label{lem:indicatorislagrangian}
    For $(\rho^j_t, F^j_t) \in \R^+ \times  \R, (\gamma^i_t, f^i_t, \pmb{\eta^i_t}) \in \R^+ \times \R \times \R^{\vert Z(V^i)\vert}$ the convex conjugates $\iota^*_{S_E}$ and $\iota^*_{S_{V^i}^\kappa}$ can be identified with the generalised Lagrangians, i.e.
    \begin{align}
        \notag
        \iota^*_{S_E}(\rho^j_t, F^j_t) &= A(\rho^j_t, F^j_t) = \left\{\begin{array}{cl} \frac{|F^j_t|^2}{2 \rho^j_t} & \text{ if } \rho^j_t > 0,\\ 0 & \text{ if } (\rho^j_t, F^j_t) = (0,0),\\ + \infty & \text{ otherwise.} \\ \end{array}\right.
        \\
        \iota^*_{S_{V^i}^\kappa}(\gamma^i_t, f^i_t, \pmb{\eta^i_t}) &= \left\{\begin{array}{cl} \kappa^2 \frac{|f^i_t|^2}{2 \gamma^i_t} & \text{ if } \gamma^i_t > 0 \text{ and } f^i_t  + \sum_{k\in Z(V^i)} \eta_t^{i,k} =0 ,\\ 0 & \text{ if } (\gamma^i_t, f^i_t +\sum_{k\in Z(V^i)} \eta_t^{i,k})  = (0,0),\\ + \infty & \text{ otherwise.} \\ \end{array}\right.
        \label{eq:indicatorlagrangianV}
    \end{align}
    where $\pmb{\eta^{i,1}_t} = \eta^{i,1}_t, \ldots, \eta_t^{i,Z(V^i)}$ denotes the vector of neighbouring edges originating or ending at the vertex $V^i$.
    The conditions $f^i_t  + \sum_{k \in Z(V^i)} \eta_t^{i,k} =0$ reflects by duality the invariance of $S_{V^i}^\kappa$ under diagonal shifts $b+kC , \pmb{c} +kC$ discussed in the previous paragraph.
    An alternative version of \eqref{eq:indicatorlagrangianV} in terms of the \textit{extended action} is given by
    \begin{align}
        \label{eq:indicatorlagrangianVextended}
        \sum_{i=1}^n\iota^*_{S_{V^i}}(\gamma^i_t, f^i_t, \pmb{\eta_t^i}) = \sum_{i=1}^nA(\gamma^i_t, f^i_t) + \begin{cases} 0 & \text{if } f^i_t +  \sum_{k\in  Z(V^i) } \eta_t^{i,k}  = 0, \\ +\infty & \text{otherwise}. \end{cases} = \bar A(\pmb{\gamma_t},\pmb{f_t}, \pmb{\eta_t}).
    \end{align}
\end{lemma}
The proof of this lemma uses calculations similar to those in \cite[Lemma 5.17]{santambrogio_optimal_2015}. In particular note that the normalisation factor $\frac{1}{|Z(V^i)|}$ that appears in the definition of $S^\kappa_{V^i}$ yields the constraint $f^i_t + \sum_{k \in Z(V^i) } \eta_t^{i,k}  = 0$.

As a next step we will define the concept of recession functions that will be needed in the following duality theorem:
\begin{definition}[recession function {\cite[Chapter 4]{rockafellar_integrals_1968}}]
    \label{def:recessionfunction}
    Let $T$ be an arbitrary set. A function $f^\infty$ on $[0,T] \times \R^n$ is called recession-function of a function $f^* \colon [0,T] \times \R^n \rightarrow \R^1 \cup \lbrace + \infty\rbrace$, if
    \begin{align*}
        f^\infty(t,w) = \lim_{\epsilon \rightarrow + \infty} [ f^*(t,x^* + \epsilon w) - f^*(t,x^*)]/ \epsilon
    \end{align*}
    whenever $x^* \in \R^n$ satisfies $f^*(t,x^*) < \infty$. 
\end{definition}

Note that we wrote $f^*$ to emphasise that this function will be a convex conjugate of a function $f$ in the following. The fact that convex-conjugates $f^*(t, \cdot)$ are lower semi-continuous, convex and not identically to $+\infty$, implies that $f^\infty(t, \cdot)$ is a well-defined, lower continuous, positively homogeneous, convex function from $\R^n$ to $\R \cup \lbrace + \infty \rbrace$, vanishing at 0, see \cite[§ 8]{rockafellar_convex_1970}.

\vspace{2ex}
To improve the readability of the following duality theorem, we define the space
\begin{align}
    \label{eq:definitionOmega}
    \mathcal{C} \coloneqq \big[ C^1(\mathcal{Q}_{\overline{E}^1}) \times \ldots \times C^1(\mathcal{Q}_{\overline{E}^m}) \big] \times \big[ C^1(\mathcal{Q}_{V^1}) \times \ldots \times C^1(\mathcal{Q}_{V^n}) \big]
\end{align}
and for $\phi^j_t \in C^1(\mathcal{Q}_{\overline{E^j}}), \psi^i_t \in C^1(Q_{V^i}), \; \pmb{\phi_t} = (\phi^1_t, \ldots , \phi^m_t), \pmb{\psi_t}= (\psi^1_t, \ldots , \psi^n_t)$ we define the primal functional 
\begin{align*}
    \mathcal{J}^\kappa(\pmb{\phi_t}, \pmb{\psi_t}) &\coloneqq \sum_{j=1}^m \int_{\overline{E^j}} \big(\phi^j_1 d \rho^j_1 - \phi^j_0 \; d \rho^j_0 \big)
    + \sum_{i=1}^n ( \psi^i_1 \gamma^i_1   - \psi^i_0 \gamma^i_0  )
    \\
    &\qquad + \sum_{j=1}^m \iint_{\mathcal{Q}_{\overline{E^j}}} \iota_S( \partial_{t} \phi^j_t, \partial_x \phi^j_t) ~ dxdt - \sum_{i=1}^n \int_0^1 \iota_{S_{V^i}^\kappa}(\partial_t \psi^i_t, \psi^i_t, \phi_t\vert_{V^i}) ~ dt.
\end{align*}
These definitions enable us to prove the duality theorem:

\begin{theorem}[Duality theorem]
    \label{thm:dualityexistence}
    Given admissible $(\pmb{\rho_0}, \pmb{\rho_1}, \pmb{\gamma_0}, \pmb{\gamma_1})$ we have the duality
    \begin{align}
        \label{eq:duality}
        \mathcal{W}_\kappa^2(\pmb{\rho_0},\pmb{\rho_1},\pmb{\gamma_0},\pmb{\gamma_1}) = \sup_\mathcal{C} \mathcal{J}^\kappa (\pmb{\phi_t}, \pmb{\psi_t})
    \end{align}
    where  $\inf = \min$ is attained in \eqref{eq:distancefunctional}.
\end{theorem}

\begin{proof}
    With $ \partial_t \pmb{\psi_t} = (  \partial_t \psi^1_t, \ldots , \partial_t \psi^n_t)$ etc., we start by defining the unfolding operator
    \begin{align*}
        L \colon \qquad \mathcal{C} ~ &\rightarrow ~ \text{Range}(L)
        \\
       (\pmb{\phi_t}, \pmb{\psi_t}) ~ &\mapsto ~ \big(\partial_t \pmb{\phi_t}, \partial_x \pmb{\phi_t} ; \partial_t \pmb{\psi_t}, \pmb{\psi_t}, (\pmb{\phi_t}^1, \ldots, \pmb{\phi_t}^n)  \big).
    \end{align*}
    with
    \begin{align}\label{eq:def_ordered_phi}
    \pmb{\phi_t}^i = \left( \left. \phi_t^j \right|_{V^i}\right)_{j \in Z(V^i)} \; .
    \end{align} 
    As pointwise restriction and differentiation are continuous operations, the unfolding operator is continuous for the natural topology on $\mathcal{C}$ 
    and 
    \begin{align*}
        \text{Range}(L) &= \Big[C(\mathcal{Q}_{\overline{E}^1}) \times \ldots \times C(\mathcal{Q}_{\overline{E}^m}) \Big]^2 \times \Big[C(\mathcal{Q}_{V^1}) \times \ldots \times C(\mathcal{Q}_{V^n})\Big] \times  \Big[C^1(\mathcal{Q}_{V^1}) \times \ldots \times C^1(\mathcal{Q}_{V^n})\big] \\
        &\quad \times \Big[ [C^1(\mathcal{Q}_{V^1})]^{|Z(V^1)|} \times \ldots \times [C^1(\mathcal{Q}_{V^n})]^{|Z(V^n)|}\Big] \;.
    \end{align*}
    With these definitions, we can express the primal problem as 
    \begin{align*}
        \sup_{(\pmb{\phi_t}, \pmb{\psi_t}) \in \mathcal{C}} \Big\lbrace \mathcal{F}(\pmb{\phi_t}, \pmb{\psi_t}) + \mathcal{G}(L(\pmb{\phi_t}, \pmb{\psi_t})) \Big\rbrace
    \end{align*}
    with
    \begin{align*}
        \mathcal{F}(\pmb{\phi_t}, \pmb{\psi_t}) &= \sum_{j=1}^m \int_{\overline{E^j}} \big(\phi^j_1 ~ d \rho^j_1 - \phi^j_0 ~d \rho^j_0 \big)
    + \sum_{i=1}^n \big(\psi^i_1 \gamma^i_1 - \psi^i_0 \gamma^i_0 \big) 
        \intertext{and}
        \mathcal{G}(L(\pmb{\phi_t}, \pmb{\psi_t})) &= - \sum_{j=1}^m \iint_{\mathcal{Q}_{\overline{E}}} \iota_S( \partial_{t} \phi^j_t, \partial_x \phi^j_t) ~ dxdt - \sum_{i=1}^n \int_0^1 \iota_{S_{V^i}^\kappa}(\partial_t \psi^i_t, \psi^i_t, \phi_t\vert_{V^i}) ~ dt.
    \end{align*}
    Note that $\mathcal{F}(\pmb{\phi_t}, \pmb{\psi_t})$, defined on the product space $\mathcal{C}$, is a sum of functions that are linear continuous with respect to their corresponding variables. Moreover, each $\iota_S$ is convex, proper and l.s.c., so both $\mathcal{F}$ and $\mathcal{G}$ are concave, proper and upper semi-continuous functionals on $\mathcal{C}$.
    To apply the Fenchel-Rockafellar theorem, we need a pair $(\pmb{\phi_t}, \pmb{\psi_t})$ such that $\mathcal{G}(L(\pmb{\phi_t}, \pmb{\psi_t}))$ is continuous at $L(\pmb{\phi_t}, \pmb{\psi_t})$ and $\mathcal{F}(\pmb{\phi_t}, \pmb{\psi_t})$ is finite. 
    An example for such a pair is given by $\pmb{\phi_t} = - \pmb{1} t, \pmb{\psi_t} =  - t$, as
    \begin{align*}
        \mathcal{F}(- \pmb{1} t, -  t) &= \sum_{j=1}^m \int_{\overline{E^j}} - d \rho^j_1  + \sum_{i=1}^n - \gamma^i_1  < \infty \qquad \text{ and}
        \\
         \mathcal{G}(L(- \pmb{1} t, -  t)) &= - \sum_{j=1}^m \iint_{\mathcal{Q}_{\overline{E^j}}} \iota_S( -1, 0) ~ dxdt - \sum_{i=1}^n \int_0^1 \iota_{S_{V^i}^\kappa}\big(- 1, - t, -1t \big) ~ dt = 0,
    \end{align*}
    where $\mathcal{G}$ is continuous at $\mathcal{G}(L(- \pmb{1} t, - t))$ as $\mathcal{G}(L(- \pmb{1} t \pm \eps, -  t \pm \eps)) = 0$ for $\eps \ll 1$.
    This pair is clearly a solution of the Hamilton-Jacobi equations \eqref{eq:Lagrangedual}.

    Thus, the Fenchel-Rockafellar theorem \ref{th:Rockafellar} guarantees that
    \begin{align}
        \label{eq:Duality1}
        \sup_\mathcal{C} \mathcal{J}^\kappa (\pmb{\phi_t}, \pmb{\psi_t}) = \inf_{\pmb{\hat{\mu}_t} \in \mathcal{C}^*} \Big\lbrace - \mathcal{F}^*(-L^*\pmb{\hat{\mu}_t}) - \mathcal{G}^*(\pmb{\hat{\mu}_t}) \Big\rbrace,
    \end{align}
    where $-\mathcal{F}^* = (- \mathcal{F})^*, - \mathcal{G}^* = (- \mathcal{G})^*$ are the Fenchel-Legendre (convex) conjugates of the convex functional $\mathcal{-F, -G}$, respectively, and $L^* \colon (\text{Range}(L))^* \rightarrow \mathcal{C}^*$ is the adjoint operator of the unfolding operator $L$ and the target dual space identifies to
    \begin{align*}
         (\text{Range}(L))^* &= \Big[ \mathcal{M}(\mathcal{Q}_{\overline{E}^1}) \times \ldots \times  \mathcal{M}(\mathcal{Q}_{\overline{E}^m}) \Big]^2 \times \Big[ \mathcal{M}(\mathcal{Q}_{V^1}) \times \ldots \times  \mathcal{M}(\mathcal{Q}_{V^n}) \Big]^3 
    \end{align*}
    with elements denoted by
    \begin{align}
        \label{eq:segmentationnu}
        \pmb{\hat{\mu}_t} = (\pmb{\hat{\mu}_{E_t}}, \pmb{\hat{\mu}_{V_t}}) = (\pmb{\rho_t}, \pmb{F_t} ; \pmb{\gamma_t}, \pmb{f_t}, \pmb{\eta} ) \in  (\text{Range}(L))^*.
    \end{align}
    Similar to Theorem 2 in \cite{monsaingeon_new_2020}, one can show that
    \begin{align}
        \label{eq:Duality2}
        &- \mathcal{F}^*(-L^*\pmb{\hat{\mu}_t}) = \sup_{\pmb{\phi_t}, \pmb{\psi_t} \in \mathcal{C}} \Big\lbrace \mathcal{F}(\pmb{\phi_t}, \pmb{\psi_t}) - \langle \pmb{\hat{\mu}_t}, L(\pmb{\phi_t}, \pmb{\psi_t})\rangle_{(\text{Range}(L))^*, \text{Range}(L)} \Big\rbrace  
        \\
        &\qquad =\begin{cases} 0 & \text{if } \begin{cases} \sum_{j=1}^m \partial_t \rho^j_t + \partial_x F^j_t = 0 \text{ in } E^j &\text{ with } F^j_t \cdot  \nu_{i,j} = - \eta_t^{i,j} \text{ and } \rho^j\vert_{t=0,1} = \rho^j_{0,1},
        \\
        \sum_{i=1}^n \partial_t \gamma^i_t = f^i_t \text{ on } V^i &\text{ with } \gamma^i\vert_{t=0,1} = \gamma^i_{0,1}, \end{cases} \\ +\infty &\text{otherwise,} \end{cases}
    \end{align}
    where $\alpha(E^j)$ denotes the initial vertex of $E^j$, $\omega(E^j)$ its terminal vertex and the equations and initial-terminal/boundary conditions should be understood in the integral sense as in Definition \ref{def:weakformcontinuityequation}.
    Moreover, for a generic element $\pmb{\zeta_t} = ( \pmb{\alpha_t}, \pmb{\beta_t} ; \pmb{a_t}, \pmb{b_t}, \pmb{c_t}) \in \text{Range}(L)$, we compute
    \begin{align*}
        - \mathcal{G}^*(\pmb{\hat{\mu}_t}) & \sup_{\pmb{\zeta_t} \in \text{Range}(L)} = \Big\lbrace \langle \pmb{\hat{\mu}_t}, \pmb{\zeta_t} \rangle_{(\text{Range}(L))^*, \text{Range}(L)} + \mathcal{G}(\pmb{\zeta_t}) \Big\rbrace
        \\
        &= \sup_{(\pmb{\alpha_t}, \pmb{\beta_t})} \Bigg\lbrace \sum_{j=1}^m \iint_{\mathcal{Q}_{\overline{E^j}}}  \alpha^j_t ~ d \rho^j_t + \sum_{j=1}^m \iint_{\mathcal{Q}_{\overline{E^j}}} \beta^j_t \cdot d F^j_t
        - \sum_{j=1}^m \iint_{\mathcal{Q}_{\overline{E^j}}} \iota_S( \alpha^j_t, \beta^j_t) ~ dxdt \Bigg\rbrace
        \\
        &+ \sup_{(\pmb{a_t},\pmb{b_t},\pmb{c_t})} \Bigg\lbrace \sum_{i=1}^n \int_0^1 a^i_t \gamma^i_t ~d t + \sum_{i=1}^n \int_0^1 b^i_t f^i_t~dt + \sum_{i=1}^n \int_0^1 \pmb{c}^{i}_t \cdot \pmb{\eta}^{i}_t ~dt
        - \sum_{i=1}^n \int_0^1 \iota_{S_{V^i}^\kappa}(a^i_t, b^i_t, \pmb{c}^i_t) ~ dt \Bigg\rbrace ,
    \intertext{where $\pmb{c}_t^i$ are defined as $\pmb{c}_t^i$ in \eqref{eq:def_ordered_phi}, we used that $(\pmb{\alpha_t}, \pmb{\beta_t})$ and $(\pmb{a_t},\pmb{b_t},\pmb{c_t})$ are uncoupled. Moreover, as each summand is uncoupled to the other others, we obtain }
        &= \sum_{j=1}^m \sup_{(\alpha^j_t, \beta^j_t)} \Bigg\lbrace \iint_{\mathcal{Q}_{\overline{E^j}}}  \alpha^j_t ~ d \rho^j_t + \iint_{\mathcal{Q}_{\overline{E^j}}} \beta^j_t \cdot d F^j_t
        - \iint_{\mathcal{Q}_{\overline{E^j}}} \iota_S( \alpha^j_t, \beta^j_t) ~ dxdt \Bigg\rbrace
        \\
        &+ \sum_{i=1}^n \sup_{(a^i_t,b^i_t,\pmb{c}^i_t)} \Bigg\lbrace \int_0^1 a^i_t \gamma^i_t ~dt + \int_0^1 b^i_t f^i_t ~dt + \int_0^1 \pmb{c}^i_t\cdot \pmb{\eta}^i_t ~dt 
        - \int_0^1 \iota_{S_{V^i}^\kappa}(a^i_t, b^i_t, \pmb{c}^i_t) ~ dt \Bigg\rbrace.
    \end{align*}
    Applying \cite[Theorem 5]{rockafellar_integrals_1968} allows us to interpret the previous equation as a sum of two convex conjugates and switch the operations of  convex conjugation with the integration.
    Next, exploiting lemma \ref{lem:indicatorislagrangian} we obtain 
    \begin{align*}
       - \mathcal{G}^*(\pmb{\hat{\mu}_t}) &= \Bigg( \sum_{j=1}^m \iint_{ Q_{\overline{E^j}}} A\Big(\frac{d \hat{\mu}_{E^j}}{d \mathcal{L}_E^j} \Big) ~d \mathcal{L}_{E^j}
        + \sum_{j=1}^m \iint_{ Q_{\overline{E^j}}} A^\infty \Big(\frac{d \hat{\mu}_{E^j}}{d \hat{\mu}_{E^j}^S} \Big) ~d \hat{\mu}_{E^j}^S \Bigg)
        + \int_{[0,1]} \overline A(\pmb{\hat{\mu}_{V_t}}) ~dt .
    \intertext{Here, $\mathcal{L}_{E^j}$ denotes the space-time Lebesgue measure on $E^j$ for a given edge.
    The measures $\hat{\mu}_{E^j}^S, \hat{\mu}_{V^i}^S$ are any non-negative measures dominating the singular parts of $\vert\hat{\mu}_{E^j}\vert, \vert \hat{\mu}_{V^i}\vert$ and $A^\infty$ denotes the recession function of $A$.
    Since $A$ is 1-homogeneous, its recession function is $A^\infty = A$, see definition \ref{def:recessionfunction} and \cite[Corollar 8.5.2]{rockafellar_convex_1970}. Then, we can write}
        - \mathcal{G}^*(\pmb{\hat{\mu}_t}) &= \sum_{j=1}^m \iint_{ Q_{\overline{E^j}}} A\Big(\frac{d \hat{\mu}_{E^j}}{d \lambda_{E^j}} \Big) d \lambda_{E^j}
        +  \int_{[0,1]} \bar A( \pmb{\hat{\mu}_{V_t}}) dt
    \end{align*}
    for any dominating measure $\lambda_{E^j} \gg \vert \hat{\mu}_{E^j} \vert$.
    Now we introduce a different segmentation for $\pmb{\hat{\mu}_t}$ and replace the notation given in \eqref{eq:segmentationnu} by \begin{align*}
       \pmb{\hat{\mu}_t} = (\pmb{\mu_{E_t}}, \pmb{\mu_{V_t}}, \pmb{\eta_t}) \coloneqq (\pmb{\rho_t}, \pmb{F_t} ; \pmb{\gamma_t}, \pmb{f_t} ; \pmb{\eta_t} )
    \end{align*}
    in order to relate $-\mathcal{G}^*(\pmb{\hat{\mu}_t})$ to the action functional \eqref{eq:actionfunctional}.
    With this choice (and by \eqref{eq:indicatorlagrangianVextended}) we obtain
    \begin{align}
        \begin{split}
       - \mathcal{G}^*(\pmb{\hat{\mu}_t}) &= \sum_{j=1}^m \iint_{ Q_{\overline{E^j}}} A\Big(\frac{d \mu_{E^j}}{d \lambda_{E^j}} \Big) ~d \lambda_{E^j}
        +  \sum_{i=1}^n \int_{[0,1]} A( \mu_{V^i})~ dt
        + \sum_{i=1}^n \begin{cases} 0 & \text{if } f^i_t + \sum_{k \in Z(V^i) } \eta_t^{i,k}  = 0 \\  +\infty & \text{otherwise} \end{cases}
        \\
        &= \mathcal{A}(\pmb{\rho_0},\pmb{\rho_1},\pmb{\gamma_0},\pmb{\gamma_1})
        + \sum_{i=1}^n \begin{cases} 0 & \text{if } f^i_t +  \sum_{k \in Z(V^i)} \eta_t^{i,k}  = 0 \\  +\infty & \text{otherwise .} \end{cases}
        \label{eq:Duality3}
        \end{split}
    \end{align}
    Combining \eqref{eq:Duality1}, \eqref{eq:Duality2} and \eqref{eq:Duality3} yields the correct flux conditions $\sum_{j \in Z(V^i)} F^j_t \cdot \nu_{i,j} = f^i_t$ (since $F^j_t \cdot  \nu_{i,j}  = - \eta^{i,j}_t$ from $-\mathcal{F}^*(-L^*\pmb{\hat{\mu}_t}) < + \infty$ and $f^i_t + \sum_{j \in Z(V^i) }\eta_t^{i,j} = 0$ from $-\mathcal{G}^*(\pmb{\hat{\mu}_t}) < + \infty$), we end up with the claimed duality
    \begin{align*}
          \sup_\mathcal{C} \mathcal{J}^\kappa (\pmb{\phi_t}, \pmb{\psi_t}) = \inf_{\pmb{\mu_t} \in \mathcal{CE}(\pmb{\rho_0},\pmb{\rho_1},\pmb{\gamma_0},\pmb{\gamma_1})} \mathcal{A} \big(\pmb{\mu_t}) = \mathcal{W}_\kappa^2(\pmb{\rho_0},\pmb{\rho_1},\pmb{\gamma_0},\pmb{\gamma_1}).
    \end{align*}
    This completes the proof as $\mathcal{W}_\kappa^2$ is finite as shown in Lemma \ref{lem:distancefinite} and the Fenchel-Rockafellar theorem assures that the supremum in \eqref{eq:duality} is attained.
\end{proof}

As expected the functional $\mathcal{W}_\kappa^2$ forms a metric. 
Using Theorem \ref{thm:dualityexistence} and Prop. \ref{prop:propertiesofthesolutions} (iii), the proof of the following proposition is analogue to \cite[Prop. 3.10]{monsaingeon_new_2020}.


\begin{proposition}[{\cite[Prop. 3.10]{monsaingeon_new_2020}}]
    \label{prop:quantityisdistance}
    The quantity $\mathcal{W}_\kappa^2(\pmb{\rho_0},\pmb{\rho_1},\pmb{\gamma_0},\pmb{\gamma_1})$, defined in \ref{def:quantityW}, is a distance on $\mathcal{P}(\mathcal{G})$, i.e. the set of probability measures on the network $\mathcal{G}$.
\end{proposition}

\section{Limiting cases and relationship to other metrics}
\label{sec:Network_relationships}

In this section we want to explore the influence of $\kappa$ on the transport problem putting special emphasis on the limiting case $\kappa \rightarrow \infty$.
This case translates to fact that the costs for transporting mass over the vertices goes to infinity.

The first step is to analyse the relationship between our metric $\mathcal{W}^2_\kappa$ and some other metrics on the graph. There are several special metrics we can consider, e.g. separate metrics on edges and nodes. Above we have already introduced the bounded Lipschitz distances on edges and nodes. The simplest case related to our metric is the Fisher-Rao metric  $\mathcal{FR}_{\kappa}$ on nodes, i.e.,
$$     \mathcal{FR}_{\kappa}^2(\pmb{\gamma_0, \gamma_1}) \coloneqq \min \mathcal{A}^\mathcal{FR} = \min \Bigg \lbrace \sum_{i=1}^n \int_0^1 \kappa^2 \frac{\vert f^i_t \vert^2}{2 \gamma^i_t} ~dt \quad \text{s.t. } \quad \partial_t \gamma^i_t  = f^i_t \text{ in } V^i \Bigg \rbrace . 
$$
Existence of minimisers for the Fisher-Rao distance is given in \cite[Thm 3.1]{chizat_interpolating_2010} which allows us to write `min' instead of `inf'.

If the masses in the edges are compatible, i.e. $ \Vert \rho_0^j \Vert = \Vert \rho_1^j \Vert $, for $j=1,\ldots, m$, we can also consider the 
the classical Wasserstein distance on each edge $\mathcal{W}^2_{E^j}$ as the problem completely decouples to 
\begin{align}\label{eq:Wasserstein_each_edge}
    \sum_{j=1}^m\mathcal{W}^2_{E^j}(\rho_0^j,\rho_1^j) &=
    \sum_{j=1}^m \min\Bigg \lbrace\iint_{\overline{E^j} \times [0,1] } \frac{\vert F^j_t \vert^2}{2 \rho^j_t} ~dxdt \quad \text{s.t. } \quad
    \begin{matrix}
        \partial_t \rho^j_t + \partial_x F^j_t = 0  & \text{ in } E^j,  \\
        F^j_t \cdot \nu_{i,j} = 0 & \text{ in } \partial \overline{E^j}
    \end{matrix}
    \Bigg \rbrace .
\end{align}

In the case of incompatible masses on individual edges, but compatible overall edge mass, i.e. 
$$ \sum_j \Vert \rho_0^j \Vert =  \sum_j \Vert \rho_1^j \Vert $$
we can introduce a Wasserstein metric on the edges only which, at the nodes, are connected via Kirchhoff's law, i.e.
$$ \mathcal{W}_{\mathcal{E}}^2 (\pmb{\rho_0, \rho_1}) = \min \Bigg \lbrace \sum_{j=1}^m \iint_{\overline{E^j} \times [0,1] } \frac{\vert F^j_t \vert^2}{2 \rho^j_t} ~dxdt \quad \text{s.t. } \quad
    \begin{matrix}
        \partial_t \rho^j_t + \partial_x F^j_t = 0  & \text{ in } E^j,  \\
     \sum_{j \in Z(V^i)} F^j_t \cdot \nu_{i,j} = 0 & \text{ in } \partial \overline{E^j} ,
    \end{matrix}
    \Bigg \rbrace $$

    The existence of a minimizer can be shown analogous to the original Wasserstein metric, since Kirchhoff's law directly allows a weak formulation of the constraint equation on the whole network of edges.
Note that for comparison the original distance with $\kappa =0$,
\begin{align*}
    \mathcal{W}_{0}^2(\varsigma_0,\varsigma_1) \coloneqq& \min \Bigg\lbrace \iint_{\mathcal{G} \times [0,1] } \frac{\vert H \vert^2}{2 \varsigma} ~dxdt \quad \text{s.t. } \quad 
    \begin{matrix} 
        \partial_t \varsigma + \partial_x H = h & \text{in } \mathcal{Q}_\mathcal{G},  \\
        \sum_{j \in Z(V^i)} F^j_t \cdot \nu_{i,j} = f^i_t & \text{in } \mathcal{Q}_\mathcal{V}.
    \end{matrix} \Bigg\rbrace ,
\end{align*}
with $\varsigma_0,\,\varsigma_1$ defined in \eqref{eq:total_mass_network} and $H, h$ as in Proposition \ref{prop:globalequation}. 

\begin{proposition}[{Sandwich theorem adapted to \cite[Prop. 5.1]{monsaingeon_new_2020}}]
    \label{prop:Sandwichtheorem}
    For any admissible initial and terminal network concentrations $(\pmb{\rho_0},\pmb{\rho_1},\pmb{\gamma_0},\pmb{\gamma_1})$, it holds that
    \begin{alignat*}{3}
        &(i) \hspace{33ex} \mathcal{FR}_{\kappa}^2(\pmb{\gamma_0, \gamma_1}) &&\leq \mathcal{W}_\kappa^2(\pmb{\rho_0},\pmb{\rho_1},\pmb{\gamma_0},\pmb{\gamma_1}) \; ,
        \\ 
        &(ii) \qquad \sum_{j=1}^m d_{\text{BL}, \overline{E^j}} (\rho^j_s, \rho^j_t ) 
            + \sum_{i=1}^n d_{\text{BL}, V^i}(\gamma^i_s, \gamma^i_t) &&\leq C_\kappa {\mathcal{W}_\kappa}(\pmb{\rho_0},\pmb{\rho_1},\pmb{\gamma_0},\pmb{\gamma_1})  \qquad \text{and}
       \intertext{
     if the node masses are compatible, i.e. 
       $\pmb{\gamma_0} = \pmb{\gamma_1}$, }
         &(iii) \hspace{27ex} \mathcal{W}_\kappa^2(\pmb{\rho_0},\pmb{\rho_1},\pmb{\gamma_0},\pmb{\gamma_1})  &&\leq \mathcal{W}_{\mathcal{E}}^2(\pmb{\rho_0, \rho_1}) \; .
    \end{alignat*}  
\end{proposition}

\begin{proof}
    \begin{enumerate}
    \item[(i)] As proven in Theorem \ref{thm:dualityexistence}, we know that $\mathcal{W}^2_\kappa$ geodesics exist and we pick any geodesic $\pmb{\hat{\mu}_t} = (\pmb{\hat{\mu}_E}, \pmb{\hat{\mu}_V}) = (\pmb{\rho_t}, \pmb{F_t} ; \pmb{\gamma_t}, \pmb{f_t}, \pmb{\eta} )$ connecting $\pmb{\rho_0}$ and $\pmb{\rho_1}$ and by definition $\mathcal{A}(\pmb{\hat{\mu}_t}) = \mathcal{W}^2_\kappa$ holds.
    For $i\in \lbrace 1, \ldots , n\rbrace$ each component $\mu_{V^i}$ of the vector $\pmb{\hat{\mu}_V}$ solves the continuity equations $\partial_t \gamma^i_t = f^i_t$, thus connecting $\gamma^i_0$ and $\gamma^i_1$ on $V^i$, and is therefore admissible for the action $\mathcal{A}^{\mathcal{FR}}$ of the Fisher-Rao metric. Thus
    \begin{align*}
         \mathcal{FR}_{\kappa}^2(\pmb{\gamma_0, \gamma_1}) = \min_{\pmb{\hat{\mu}_V}'} \mathcal{A}^{\mathcal{FR}}(\pmb{\hat{\mu}_V}') \leq \mathcal{A}^{\mathcal{FR}}(\pmb{\hat{\mu}_V}) \leq \mathcal{A}(\pmb{\hat{\mu}_t}) = \mathcal{W}^2_\kappa(\pmb{\rho_0},\pmb{\rho_1},\pmb{\gamma_0},\pmb{\gamma_1}) \; ,
    \end{align*}
    which shows the assertion.
    
    \item[(ii)] Applying Theorem \ref{thm:dualityexistence}, we choose a geodesic $\pmb{\mu_t} \in \mathcal{CE}(\pmb{\rho_0},\pmb{\rho_1},\pmb{\gamma_0},\pmb{\gamma_1})$. By Proposition \ref{prop:propertiesofthesolutions} (iii) we obtain
    \begin{align*}
         \sum_{j=1}^m d_{\text{BL}, \overline{E^j}} (\rho^j_s, \rho^j_t ) 
            + \sum_{i=1}^n d_{\text{BL}, V^i}(\gamma^i_s, \gamma^i_t) \leq C_\kappa \sqrt{\mathcal{A}(\pmb{\mu_t})} \vert t-s\vert^{\frac{1}{2}} = C_\kappa \mathcal{W}_\kappa^2(\pmb{\rho_0},\pmb{\rho_1},\pmb{\gamma_0},\pmb{\gamma_1}) \; .
    \end{align*}

    \item[(iii)]
    Pick an interior Wasserstein geodesic $(\pmb{\rho_t}, \pmb{F_t})$ and the trivial boundary geodesic $(\pmb{\gamma_0},\pmb{0})$. Together these geodesics form a solution of the generalised continuity equation in the sense of Definition \ref{def:weakformcontinuityequation}. 
    Consequently, $\pmb{\mu_t} \coloneqq (\pmb{\rho_t}, \pmb{F_t},\pmb{\gamma_t},\pmb{0})$ forms an admissible candidate in \ref{def:quantityW} and we obtain
    \begin{align*}
        \mathcal{W}_\kappa^2(\pmb{\rho_0},\pmb{\rho_1},\pmb{\gamma_0},\pmb{\gamma_1}) \leq \mathcal{A}(\pmb{\mu_t}) = \mathcal{W}_{\mathcal{E}}^2(\pmb{\rho_0, \rho_1}) \; .
    \end{align*}
    \end{enumerate}
\end{proof}
We can also show that, unless the initial and final mass on the vertices is zero, every geodesic is such that for every edge with unequal initial and final mass, the mass of at least one of the attached vertices must change.
\begin{proposition}
     For any admissible initial and terminal network concentrations $(\pmb{\rho_0},\pmb{\rho_1},\pmb{\gamma_0},\pmb{\gamma_1})$ such that $\pmb{\gamma_0}=\pmb{\gamma_1}\neq 0$. Then for every edge $j$ with $\rho_0^{j} \neq \rho_1^{j}$ either $f^{\bar\alpha(j)}$ or $f^{\bar \omega(j)}$ must be nonzero. 
     
\end{proposition}
\begin{proof}
    We argue by contradiction and use that the dual potentials $(\pmb{\phi_t},\pmb{\psi_t})$ are (smooth) subsolutions of the Hamilton-Jacobi equations \eqref{eq:Lagrangedual} and satisfy the relations
    \begin{align*}
            \frac{F^j_t}{\rho^j_t} = \partial_x \phi^j_t\quad \text{ and } \quad \kappa^2 \frac{f^i_t}{\gamma^i_t} = \psi^i_t - \frac{1}{|Z(V_i)|}\sum_{j\in Z(V_i)}\phi^j_t,
    \end{align*}
    for $i=1,\ldots, n$. Fix and edge $j^* \in \{1, \ldots, m\}$ with $\rho_0^{j} \neq \rho_1^{j}$ and assume $f_t^{\bar\alpha(j^*)} = 0$ as well as $f_t^{\bar\omega(j^*)} = 0$ for a.e. $t \in (0,1)$. This implies, for  $i^* \in \{\bar\alpha(j^*), \bar\omega(j^*)\}$, that $\gamma_t^{i^*}=C$ with $C>0$ due to $\gamma_0^{i^*}=\gamma_1^{i^*}\neq 0$. Together with
    \begin{align}\label{eq:calc_vertex_geodesic}
        0 = \kappa^2 f_t^{i^*} = \gamma_t^{i^*}\left(\psi^{i^*}_t  - \frac{1}{|Z(V_{i^*})|}\sum_{j\in Z(V_{i^*})}\phi^j_t\right),
    \end{align}
    this yields $\psi^{i^*}_t  - \frac{1}{|Z(V_{i^*})|}\sum_{j\in Z(V_{i^*})}\phi^j_t = 0$. Inserting this into \eqref{eq:Lagrangedual} yields $\partial_t \psi_t^{i^*} = 0$ and, differentiating \eqref{eq:calc_vertex_geodesic} in time, we obtain
    \begin{align*}
        0 = \partial_t \sum_{j\in Z(V_{i^*})}\phi^j_t = \frac{1}{2}\sum_{j\in Z(V_{i^*})}|\partial_x\phi^j_t|^2.
    \end{align*}
    Therefore the fluxes $F_t^{j}$ vanish for all edges $j$ connected to either one of the two vertices $\bar\alpha(j^*)$ and $\bar\omega(j^*$. In particular $F_t^{j^*}\cdot \nu_{i^*,j^*}=0$ for $i^* \in \{\bar\alpha(j^*), \bar\omega(j^*)\}$ meaning that on $j^*$, the continuity equation has no-flux boundary conditions. This is a contradiction to $\rho_0^{j^*} \neq \rho_1^{j^*}$. 
\end{proof}
\begin{proposition}[{adapted to \cite[Prop. 6.1]{monsaingeon_new_2020}}]
    For fixed $(\pmb{\rho_0},\pmb{\rho_1},\pmb{\gamma_0},\pmb{\gamma_1})$ the map $\kappa \mapsto \mathcal{W}^2_\kappa(\pmb{\rho_0},\pmb{\rho_1},\pmb{\gamma_0},\pmb{\gamma_1})$ is non-decreasing.
\end{proposition}

\begin{proof}
    Note that the set of smooth subsolutions
    \begin{align*}
        S^\kappa \coloneqq \Bigg\lbrace (\pmb{\phi_t}, \pmb{\psi_t}) \in \mathcal{C} \colon \qquad
        \sum_{j=1}^m \partial_t \phi^j_t + \frac{1}{2} \vert \partial_x \phi^j_t \vert ^2 \leq 0 
        \quad \text{and} \quad
        \sum_{i=1}^n \partial_t \psi^i_t + \frac{1}{2 \kappa^2} \Big\vert \psi^i_t - \frac{1}{|Z(V^i)|}\sum_{j\in Z(V^i)} \phi^j_t \Big\vert^2 \leq 0 \Bigg\rbrace 
    \end{align*}
    is non-decreasing in $\kappa$. The monotonicity immediately follows from the duality in Theorem \ref{thm:dualityexistence}, i.e.
    \begin{align*}
        \mathcal{W}^2_\kappa(\pmb{\rho_0},\pmb{\rho_1},\pmb{\gamma_0},\pmb{\gamma_1}) = \sup_{(\pmb{\phi_t}, \pmb{\psi_t}) \in S^\kappa} \Big\lbrace \sum_{j=1}^m \int_{\overline{E^j}} \big(\phi^j_1 ~ d \rho^j_1 - \phi^j_0 ~d \rho^j_0 \big)
        + \sum_{i=1}^n \psi^i_1 \gamma^i_1 - \psi^i_0  \gamma^i_0 \Big\rbrace ~ .
    \end{align*}
\end{proof}

Now, let us finally discuss the limit case $\kappa \rightarrow \infty$. 
In this case, transporting mass over the vertices becomes more and more costly and, intuitively, we expect that in the limit transporting mass over the vertices is prohibited.

In more detail, we recover two interesting scenarios depending on the initial concentration and the terminal concentration on the network.
The first scenario is the case where the masses are incompatibility, i.e. $\Vert \rho^j_0 \Vert \neq \Vert \rho^j_1 \Vert$ for one $1 \le j \le m$ or $\pmb{\gamma_0^i} \neq \pmb{\gamma_1^i} $ for a given $1 \le i \le n$.
More precisely, this means that there exists at least one edge or one vertex where the corresponding initial mass is different than the terminal one. 
As mass cannot be transported over the vertices, we expect that the distance functional $\mathcal{W}_\kappa^2$ goes to infinity in this case.
In the second scenario the masses are compatible. As we shall see, the distance functional will converge to a sum of classical Wasserstein metrices on each edge.
In order to proof this assumption, we will need a lemma for each scenario. 
We start with the lemma for the incompatible scenario:

\begin{lemma}[{adapted to \cite[Prop. 6.2]{monsaingeon_new_2020}}]
    \label{lem:incombatiblemasseslimitinfinity}
    For any initial and terminal network concentrations $(\pmb{\rho_0},\pmb{\rho_1},\pmb{\gamma_0},\pmb{\gamma_1})$ and any geodesic $(\pmb{\rho_t}, \pmb{F_t} ; \pmb{\gamma_t}, \pmb{f_t})$, there holds
    \begin{align*}
        \mathcal{W}^2_\kappa(\pmb{\rho_0},\pmb{\rho_1},\pmb{\gamma_0},\pmb{\gamma_1}) \geq \frac{\kappa^2}{2} \sum_{i=1}^n \Vert f^i_t \Vert^2 \geq \frac{\kappa^2}{2} \sum_{i=1}^n \Big\vert  \gamma^i_1 - \gamma^i_0 \Big\vert^2 \; .
    \end{align*}
\end{lemma}

\begin{proof}
    Applying Theorem \ref{thm:dualityexistence} allows us to choose a geodesic $\pmb{\hat{\mu}_t} = (\pmb{\rho_t}, \pmb{F_t} ; \pmb{\gamma_t}, \pmb{f_t}, \pmb{\eta_t} )$ from $\pmb{\rho_0}$ to $\pmb{\rho_1}$ and by definition we obtain
    \begin{align*}
        \mathcal{W}^2_\kappa(\pmb{\rho_0},\pmb{\rho_1},\pmb{\gamma_0},\pmb{\gamma_1}) = \mathcal{A}(\pmb{\mu_t}) 
        \geq \frac{\kappa^2}{2} \sum_{i=1}^n \int_0^1 \frac{\vert f^i_t \vert^2}{\gamma^i_t} ~ dt 
        \geq \frac{\kappa^2}{2} \sum_{i=1}^n \Big( \int_0^1 f^i_t ~ dt \Big)^2
        = \frac{\kappa^2}{2} \sum_{i=1}^n \Vert f^i_t \Vert^2 \; .
    \end{align*}
    where we used Proposition \ref{prop:propertiesofthesolutions} (ii), i.e. $ \gamma^i_t \leq \Vert \varsigma \Vert =1$, as $\varsigma$ is a probability measure, and Jensen's inequality.
    The continuity equation $\partial_t \gamma^i_t = f^i_t$ finally controls the mass difference as (again using Jensen)
    \begin{align*}
        \Vert f^i_t \Vert \geq \Big\vert \int_0^1 1 ~ d f^i_t \Big\vert \overset{\text{Def.} \ref{def:weakformcontinuityequation}}{=} \Big\vert \gamma^i_1 - \gamma^i_0 \Big\vert \; .
    \end{align*}
\end{proof}

The following result helps to solve weighted optimisation problems and will be applied in the proof of the next proposition in the case of compatible masses:

\begin{lemma}[{\cite[Lemma 6.3]{monsaingeon_new_2020}}]
    \label{lem:Hilfslemma}
    Let $K$ be a compact set, take $\mathfrak{f}, \mathfrak{g} \colon K \rightarrow \R^+ \cup \lbrace + \infty \rbrace$ two proper, lower semi-continuous functions, and consider
    \begin{align*}
       \mathfrak{h}_\kappa (x) \coloneqq \mathfrak{f}(x) + \kappa^2 \mathfrak{g}(x), \qquad \kappa >0.
    \end{align*}
    Assume that for all $\kappa$ there is a minimiser $x_\kappa \in K$ of $\mathfrak{h_\kappa}$. Then as $\kappa \rightarrow + \infty$ any cluster point $x_*$ of $\lbrace x_\kappa \rbrace$ minimises $\mathfrak{f}$ in $\mathrm{argmin}\, \mathfrak{g}$.
\end{lemma}

\begin{proposition}[{adapted to \cite[Thm. 6]{monsaingeon_new_2020}}]
    For fixed $(\pmb{\rho_0},\pmb{\rho_1},\pmb{\gamma_0},\pmb{\gamma_1})$ we obtain
    \begin{align}
        \label{eq:convergencedistance}
        \mathcal{W}^2_\kappa(\pmb{\rho_0},\pmb{\rho_1},\pmb{\gamma_0},\pmb{\gamma_1}) \underset{\kappa \; \rightarrow \; + \infty}{\rightarrow} \begin{cases} \mathcal{W}_{\mathcal{E}}^2(\pmb{\rho_0},\pmb{\rho_1})  & \text{if } \pmb{\gamma_0} = \pmb{\gamma_1}, \\ +\infty & \text{otherwise,} \end{cases}
    \end{align}
Let moreover $\pmb{\mu_t}^\kappa = (\pmb{\rho_t}^\kappa, \pmb{F_t}^\kappa ; \pmb{\gamma_t}^\kappa, \pmb{f_t}^\kappa)$ be any $\mathcal{W}^2_\kappa(\pmb{\rho_0},\pmb{\rho_1},\pmb{\gamma_0},\pmb{\gamma_1})$-geodesic.
    If the node masses are compatible, i.e.  $\pmb{\gamma_0} = \pmb{\gamma_1}$, $\Vert \rho^j_0 \Vert = \Vert \rho^j_1 \Vert$ for all $j \in \lbrace 1, \ldots , m \rbrace$
     then up to a subsequence
    \begin{align*}
        (\pmb{\rho_t}^\kappa, \pmb{F_t}^\kappa ) \rightarrow (\pmb{\rho_t}, \pmb{F_t}), \quad \text{and} \quad \Vert f_t^\kappa \Vert \rightarrow 0
    \end{align*}
    where $(\pmb{\rho_t}, \pmb{F_t})$ is a $\mathcal{W}_{\mathcal{E}}^2$-geodesic .
\end{proposition}

\begin{proof}
    For incompatible node masses, i.e.
    $\pmb{\gamma_0} \neq \pmb{\gamma_1}$, we conclude from Lemma \ref{lem:incombatiblemasseslimitinfinity} that
    \begin{align*}
        \mathcal{W}^2_\kappa (\pmb{\rho_0},\pmb{\rho_1},\pmb{\gamma_0},\pmb{\gamma_1}) \underset{\kappa \; \rightarrow + \infty}{\rightarrow} + \infty \; . 
    \end{align*}
    In the case of compatible node masses, we can apply Proposition \ref{prop:Sandwichtheorem} and Lemma \ref{lem:incombatiblemasseslimitinfinity} to control the boundary flux term  via
    \begin{align*}
        \sum_{i=1}^n \Vert f^\kappa \Vert^2 \leq \frac{2}{\kappa^2} \mathcal{W}^2_\kappa (\pmb{\rho_0},\pmb{\rho_1},\pmb{\gamma_0},\pmb{\gamma_1}) \leq \frac{2}{\kappa^2} \mathcal{W}_{\mathcal{E}}^2(\pmb{\rho_0},\pmb{\rho_1}) \rightarrow 0.
    \end{align*}
    To control the momentum $\pmb{F}^\kappa$, we observe from Proposition \ref{prop:Sandwichtheorem} that any geodesic satisfies
    \begin{align}
        \label{eq:BoundaryofFlow}
        \frac{1}{2} \sum_{j=1}^m \iint_{\overline{E^j} \times [0,1]} \frac{\vert F^{j,\kappa}_t \vert^2}{\rho^{j,\kappa}_t} ~dxdt \leq  \mathcal{W}^2_\kappa (\pmb{\rho_0},\pmb{\rho_1},\pmb{\gamma_0},\pmb{\gamma_1}) \leq  \mathcal{W}_{\mathcal{E}}^2  (\pmb{\rho_0},\pmb{\rho_1}) 
    \end{align}
    uniformly in $\kappa >0$. With $\pmb{F_t}^\kappa = \pmb{u_t}^\kappa \pmb{\rho_t}^\kappa$ and using Jensen's inequality
    \begin{align*}
        \Vert \pmb{F_t}^\kappa \Vert^2 = \sum_{j=1}^m \Big( \iint_{\overline{E^j} \times [0,1]} \vert u^\kappa_t \vert d\rho^\kappa_t \Big)^2
        \leq \sum_{j=1}^m \iint_{\overline{E^j} \times [0,1]} \vert u^\kappa_t \vert^2 ~ d\rho^\kappa_t
        = \sum_{j=1}^m \iint_{\overline{E^j} \times [0,1]} \frac{\vert F^{j,\kappa}_t \vert^2}{\rho^{j,\kappa}_t} ~ dxdt \; .
    \end{align*}
    Therefore $\pmb{F_t}^\kappa$ is bounded by $\Vert \pmb{F_t}^\kappa \Vert \leq \mathcal{W}^2_{\mathcal{E}} (\pmb{\rho_0},\pmb{\rho_1}) = C$ using \eqref{eq:BoundaryofFlow}, where $C$ does not dependent on $\kappa$. 
    Moreover $\pmb{\rho_t}^\kappa, \pmb{\omega_t}^\kappa \leq 1$, and the geodesic $\pmb{\mu_t}^\kappa$ is bounded.
    
    Our aim now is to apply Lemma \ref{lem:Hilfslemma}, so we need compactness which is guaranteed by Prokhorov's Theorem:
    In a Polish space, bounded measures are tight ($\mathcal{G}$ is a Polish space as a closed subset of the Polish space $\R^2$ ).
    Thus we can apply Prokhorov's theorem and we obtain the narrow compactness $(\pmb{\rho_t}^\kappa_t, \pmb{F_t}^\kappa_t ; \pmb{\gamma_t}^\kappa_t , \pmb{f_t}^\kappa_t) \overset{*}{\rightharpoonup} (\pmb{\rho_t}, \pmb{F_t} ; \pmb{\gamma_t}, \pmb{0})$ up to subsequences, and we only have to proof that the limit $(\pmb{\rho_t}, \pmb{F_t})$ is a $\mathcal{W}_{\mathcal{E}}^2$-geodesic.
    
    We now use Lemma \ref{lem:Hilfslemma} and set $K$ to be the set of all geodesics for all values of $\kappa > 1$, which is narrowly compact by the previous discussion and because the linear continuity equations \eqref{eq:weakformualtionedges} and \eqref{eq:weakformualtionvertices} are stable under narrow limits.
    The functions
    \begin{align*}
        \mathfrak{f} (c, \pmb{{}f_t}) \coloneqq \frac{1}{2} \sum_{j=1}^m \iint_{\overline{E^j} \times [0,1]} \frac{\vert F^j_t \vert^2}{\rho^j_t} ~dxdt \quad \text{and} \quad
        \mathfrak{g} (\pmb{\rho_t}, \pmb{F_t}, \pmb{\gamma_t}, \pmb{f_t}) \coloneqq \frac{1}{2} \sum_{i=1}^n \int_0^1 \frac{{f^i_t}^2}{\gamma^i_t} ~dt
    \end{align*}
    are convex, proper, semi-continuous w.r.t. the narrow convergence of measures \cite[Theorem 3.3]{bouchitte_new_1990}, and geodesics are minimisers of $\mathcal{A} = \mathfrak{f} + \kappa^2 \mathfrak{g}$.
    The definition of $\frac{{f^i_t}^2}{\gamma^i_t}$ in the extended sense implies that the minimisers of $\mathfrak{g}$ are solutions of the continuity equations \eqref{eq:weakformualtionedges} and \eqref{eq:weakformualtionvertices} of the form $\pmb{\mu_t} = (\pmb{\rho_t}, \pmb{F_t} ; \pmb{\gamma_t}, \pmb{0})$ and assign the value $\iint \frac{{f^i_t}^2}{\gamma^i_t} = 0$ to $\mathfrak{g}$.
    This set of solutions of $\mathcal{CE}(\pmb{\rho_0},\pmb{\rho_1},\pmb{\gamma_0},\pmb{\gamma_1})$ with $f^i_t = 0$ obviously identifies with the whole set of pairs $(\pmb{\rho_t}, \pmb{F_t})$ of independent solutions of $\partial_t \rho^j_t + \partial_x F^j_t$ for all $j$ with Kirchhoff's law 
    $$\partial_t \gamma^i_t =  \sum_{j \in Z(V^i)} F^j_t \cdot \nu_{i,j} = 0$$ at the boundary. 
    Consequently, the limit-geodesic $(\pmb{\rho_t}, \pmb{F_t})$ is a geodesic for $\mathcal{W}_{\mathcal{E}}^2$.
    
    Finally, let us analyse the convergence in distance \eqref{eq:convergencedistance}. Applying Proposition \ref{prop:Sandwichtheorem} (iii) and the lower semi-continuity of the actions with $(\pmb{\rho_t}^\kappa, \pmb{F_t}^\kappa ) \overset{*}{\rightharpoonup} (\pmb{\rho_t}, \pmb{F_t}) $ gives that
    \begin{align*}
        \limsup_{\kappa \rightarrow + \infty} \mathcal{W}^2_\kappa (\pmb{\rho_0},\pmb{\rho_1},\pmb{\gamma_0},\pmb{\gamma_1}) &\leq \mathcal{W}_{\mathcal{E}}^2 (\pmb{\rho_0},\pmb{\rho_1}) =  \sum_{j=1}^m \iint_{\overline{E^j} \times [0,1]} \frac{\vert F^j_t\vert^2}{2 \rho^j_t} ~dxdt \leq \liminf_{\kappa \rightarrow + \infty}  \sum_{j=1}^m \iint_{\overline{E^j} \times [0,1]} \frac{\vert F^{j,\kappa}_t \vert^2}{2\rho^{j,\kappa}_t} ~dxdt
        \\
        &\leq \liminf_{\kappa \rightarrow + \infty} \Big( \iint_{\overline{E^j} \times [0,1]} \frac{\vert F^{j,\kappa}_t \vert^2}{2\rho^\kappa_t}  ~dxdt + \int_0^1 \kappa^2 \frac{\vert f^{j,\kappa}_t \vert^2}{2\gamma^\kappa_t}  ~dt \Big)
        = \liminf_{\kappa \rightarrow + \infty} \mathcal{W}^2_\kappa (\pmb{\rho_0},\pmb{\rho_1},\pmb{\gamma_0},\pmb{\gamma_1}) ,
    \end{align*}
    where we used that $(\pmb{\rho_t}^\kappa, \pmb{F_t}^\kappa ; \pmb{\gamma_t}^\kappa, \pmb{f_t}^\kappa)$ is a geodesic in the last equality. Thus, $\liminf = \limsup = \lim$ in the previous chain of inequalities, which finishes the proof.
\end{proof}

\section{Gradient Flows}

Given the metric structure it appears natural to study corresponding gradient flows, which opens various questions for future research beyond the scope of this paper. Here will only provide a formal derivation of the gradient flow equations for energy functionals of the form
\begin{equation}
    {\cal E}(\pmb{\rho},\pmb{\gamma}) =  \sum_{j=1}^m {\cal G}_j(\rho^j) + \sum_{j=1}^n {\cal H}_i(\gamma^i).
\end{equation}
The gradient flow structures can be derived in a standard way from a minimizing movement scheme, as in \cite{ambrosio_2008,santambrogio2017euclidean}, which constructs a time-discrete sequence $(\pmb{\rho}^{\tau,k},\pmb{\gamma}^{\tau,k})$, $k \in \N$, whose iterates are obtained by minimizing the functional 
$$ \frac{1}{2\tau} {\cal W}_\kappa^2(\pmb{\rho}^\tau,\pmb{\rho}^{\tau,k-1},\pmb{\gamma}^\tau,\pmb{\gamma}^{\tau,k-1}) + {\cal E}(\pmb{\rho}^\tau,\pmb{\gamma}^\tau) $$
with respect to $(\pmb{\rho}^\tau,\pmb{\gamma}^\tau)$ for a given pair $(\pmb{\rho}^{\tau,k-1},\pmb{\gamma}^{\tau,k-1})$,
By a formal limit procedure this yields that the gradient flow satisfies the transport problem with fluxes 
\begin{align*}
F_t^j &= \rho_t^j \partial_x \phi_t^j \\
f_t^i &= \kappa^{-2} \gamma_t^i (\psi_t^i - \frac{1}{|Z(V_i)|}\sum_{j \in Z(V_i)}\phi^j_t) 
\end{align*}
and dual potentials
\begin{equation*}
 \phi_t^j = - {\cal G}_j'(\rho_t^j) , \qquad \psi_t^i = - {\cal H}_i'(\gamma_t^i).
\end{equation*}
This implies that the gradient flows are of the form
\begin{align*}
    \partial_t \rho_t^j &= \partial_x(\rho_t^j \partial_x {\cal G}_j'(\rho_t^j)) \\
    \partial_t \gamma_t^i &= - \kappa^{-2} \gamma_t^i ( {\cal H}_i'(\gamma_t^i) - \frac{1}{|Z(V_i)|}\sum_{j \in Z(V_i)}{\cal G}_j'(\rho_t^j)) 
\end{align*}
with the Kirchhoff condition 
\begin{equation*}
\sum_{j \in Z(V_i)} \rho_t^j \partial_x {\cal G}_j'(\rho_t^j) \cdot \nu_{i,j}  = \kappa^{-2} \gamma_t^i ( {\cal H}_i'(\gamma_T^i) - \frac{1}{|Z(V_i)|}\sum_{j \in Z(V_i)}{\cal G}_j'(\rho_t^j))
\end{equation*}
and the continuity
\begin{equation*}
   {\cal G}_j'(\rho_t^j) = {\cal G}_k'(\rho_t^k) \qquad \forall j,k \in Z(V_i).
\end{equation*}
A surprising effect of the coupling condition is the one-sided coupling of the variations of the edge energies, whose trace appears in the vertex equations. On the other hand we have the standard form as in Wasserstein gradient flows on the edges.

\begin{example}
Let us consider the standard case of a drift-diffusion equation on the edges, i.e. 
$$ {\cal G}_j(\rho^j) = \int_{E^j} \rho^j \log \rho^j  + \rho^j W_j)~dx $$
with some potential $W_j$, and some cost on the vertex concentrations
${\cal H}_i(\gamma^i) = h_i(\gamma^i)$.
\begin{align*}
    \partial_t \rho_t^j &= \partial_{xx} \rho_t^j + \partial_x(\rho_t^j W_j') \\
    \partial_t \gamma_t^i &= - \kappa^{-2} \gamma_t^i (h_i'(\gamma_t^i) - \frac{1}{|Z(V_i)|}\sum_{j \in Z(V_i)} (1+\log \rho_t^j +W_j)) 
\end{align*}
with the Kirchhoff condition 
$$
\sum_{j \in Z(V_i)} (\partial_{x} \rho_t^j + \rho_t^j W_j) \cdot \nu_{i,j}  = \kappa^{-2} \gamma_t^i ( {\cal h}_i'(\gamma_T^i) - \frac{1}{|Z(V_i)|}\sum_{j \in Z(V_i)}(1+\log \rho_t^j +W_j)).
$$
The continuity condition in the vertices becomes 
$ \log \rho_t^j +W_j = \log \rho_t^k +W_k, $
which can be reformulated as the linear transmission condition
$$ \rho_t^j e^{W_j} = \rho_t^k e^{W_k} \qquad \forall j,k \in Z(V_i).$$
 \end{example}

\vspace{2ex}
\noindent \textbf{Acknowledgments:} MB acknowledges support by the German science foundation DFG through SFB TRR 154, subproject C06. IH acknowledges support by EXC 2044 Mathematics Münster, Cluster of Excellence, Münster, funded by the German science foundation DFG.

\appendix 

\section{Formal derivation of the first order optimality conditions}
\label{sec:derivationHJ}

Based on the Lagrange calculus we present a formal derivation of the optimality conditions of \eqref{eq:minimizaionprob} which include the Hamilton-Jacobi equations already introduced in \eqref{eq:Lagrangedual}.

Using the notation $\pmb{\phi_t} = (\phi^1_t,\ldots, \phi^m_t)$ as well as $\pmb{\psi_t} = (\psi_t^1,\ldots, \psi^n_t)$ and $\pmb{\lambda_t} = (\lambda^1_t,\ldots, \lambda^n_t)$, we introduce the Lagrange functional  
\begingroup
\allowdisplaybreaks
\begin{align*}
    &\mathcal{L} \big(\pmb{\rho_t}, \pmb{\gamma_t}, \pmb{F_t}, \pmb{f_t},\pmb{\phi_t}, \pmb{\psi_t}, \pmb{\lambda_t}\big) \\
    &=
    \sum_{j=1}^m \iint_{\overline{E^j} \times [0,1]} \frac{|F^j_t|^2}{2\rho^j_t}\;dxdt + \kappa^2\sum_{i=1}^n  \int_{[0,1] } \frac{|f^i_t|^2}{2\gamma^i_t}\; dt
    + \sum_{j=1}^m\iint_{\overline{E^j} \times[0,1]} (\partial_t \rho^j_t + \partial_x F^j_t) \phi^j_t ~ dxdt \\
    &+ \sum_{i=1}^n\int_0^1 ( \partial_t \gamma_t^i - f_t^i) \psi^i_t ~dt
    + \sum_{i=1}^n\int_0^1 (f^i_t - \sum_{j \in Z(V^i)} F^j_t(V_i)  \nu_{i,j})\lambda^i_t ~dt \\ 
    &=
    \sum_{j=1}^m \iint_{\overline{E^j} \times [0,1]} \frac{|F^j_t|^2}{2\rho^j_t}\;dxdt + \kappa^2\sum_{i=1}^n  \int_{[0,1] } \frac{|f^i_t|^2}{2\gamma^i_t}\; dt
    + \sum_{j=1}^m\iint_{\overline{E^j} \times[0,1]} (- \rho^j_t\partial_t\phi^j_t -  F^j_t \partial_x \phi^j_t) ~ dxdt \\
    &    + \sum_{j=1}^m \int_{\overline{E^j}} (\rho_1^j\phi^j_t(x,1) - \rho_0^j\phi^j_t(x,0))~dx
    + \sum_{i=1}^n\sum_{j\in Z(V_i)}^m\int_0^1 F^j_t(V_i) \nu_{ij}\phi^j_t ~dt  
    \\
    &+ \sum_{i=1}^n\int_0^1 ( -\gamma_t^i\partial_t \psi^i_t - f_t^i \psi^i_t) ~dt + \sum_{i=1}^n(\gamma_1^i\psi^i_t(1) - \gamma_0^i\psi^i_t(0))\\
    &+ \sum_{i=1}^n\int_0^1 (f^i_t - \sum_{j \in Z(V^i)} F^j_t(V_i)  \nu_{i,j})\lambda^i_t ~dt.
\end{align*}
\endgroup
Calculating the gradient of  $\mathcal{L}$ into direction $\pmb{\varphi} = [\pmb {\varphi_{\rho_t}}, \pmb{\varphi_{F_t}},\pmb{\varphi_{f_t}}, \pmb{\varphi_{\gamma_t}} ]$  yields
\begingroup
\allowdisplaybreaks
\begin{align*}
    &\nabla_{\pmb{\rho_t}, \pmb{F_t}, \pmb{f_t}, \pmb{\gamma_t}} \mathcal{L} [\pmb{\varphi}] 
    \\
    &=
    \sum_{j=1}^m\iint_{\overline{E^j}\times[0,1]} \big( - \frac{\vert F^j_t \vert^2}{2 (\rho^j_t)^2} - \partial_t \phi^j_t \big) \varphi_{\rho^j_t} ~dxdt+ \sum_{j=0}^m\int_{\overline{E^j}} (\varphi_{\rho_t^j}(x,1) \phi^j_t(x,1) - \varphi_{\rho_t^j}(x,0)\phi^j_t(x,0)) ~dx
    \\
     &\quad 
    + \sum_{j=1}^m\iint_{\overline{E^j} \times[0,1]} \Big( \frac{F^j_t}{\rho^j_t} - \partial_x \phi^j_t \Big) \varphi_{F^j_t} ~dxdt  + \sum_{i=1}^n\sum_{j\in Z(V_i)}\int_0^1 (\phi^j_t + \lambda^i_t) \varphi_{F^i_t}  \nu_{ij} ~dt  \\
    &\quad + \sum_{i=1}^n\int_0^1 \Big( \kappa^2 \frac{f^i_t}{\gamma^i_t} - \psi^i_t + \lambda^i_t \Big) \varphi_{f^i_t} ~dt\\
    &\quad +  \sum_{i=1}^n \int_0^1 \Big( - \kappa^2\frac{\vert f_t^i \vert^2}{2 (\gamma_t^i)^2} - \partial_t \psi^i_t \Big) \varphi_{\gamma_t^i}  ~dt + \sum_{i=1}^n (\varphi_{\gamma_1^i}\psi^i_t(1) - \varphi_{\gamma_0^i}\psi^i_t(0))
    \; .
\end{align*}
\endgroup
Choosing the direction $\pmb{\varphi} = [\pmb{0}, \pmb{\varphi_{F_t}},\pmb{0}, \pmb{0}]$ s.t. $F_t^i\nu_{ij}=0$ for all $i\in\{0,\ldots, n\}$, $j \in Z(V_i)$ as well as $\pmb{\varphi} = [\pmb{0}, \pmb{0}, \pmb{\varphi_{f_t}},\pmb{0}]$ and setting the respective derivative to zero we obtain
\begin{align*}
    \frac{F^j_t}{\rho^j_t} - \partial_x \phi^j_t=0\quad \text{ and } \quad \kappa^2 \frac{f^i_t}{\gamma^i_t} = \psi^i_t - \lambda^i_t,
\end{align*}
for $j=1,\ldots m$ and $i=1,\ldots n$. 
From $\pmb{\varphi} = [\pmb{0}, \pmb{\varphi_{F_t}},\pmb{0}, \pmb{0}]$ s.t. $F_t^i\nu_{ij}=0$ s.t. $\varphi_{F^i_t}\nu_{ij} = \varphi_{F^i_t}\nu_{ij'}$ for all $j,j' \in Z(V_i)$ we have
\begin{align*}
    \lambda^i_t = - \frac{1}{|Z(V_i)|}\sum_{j\in Z(V_i)}\phi^j_t,\quad i=1,\ldots,n.
\end{align*}
Finally, the choices $\pmb{\varphi} = [\pmb{\varphi_{\rho_t}},\pmb{0} ,\pmb{0}, \pmb{0}]$, supported in the interior, as well as $\pmb{\varphi} = [\pmb{0}, \pmb{0},\pmb{0}, \pmb{\varphi_{\gamma_t}} ]$ result in the equations
\begin{align*}
\partial_t \phi^j_t + \frac{1}{2}|\partial_x\phi^j_t|^2 = 0,\quad \text{and} \quad \partial_t\psi^i_t  + \frac{1}{2\kappa^2}\left|\psi^i_t - \frac{1}{|Z(V_i)|}\sum_{j \in Z(V_i)}\phi^j_t\right|^2 = 0.
\end{align*}
This corresponds exactly to \eqref{eq:Lagrangedual}.

\bibliographystyle{plain}
\bibliography{references}
\end{document}